\documentclass{commat}

\title[Center of mass technique and affine geometry]{%
    Center of mass technique and affine geometry \\
    {\small To the memory of Sergei Duzhin}
    }

\author{%
    Askold Khovanskii
    }

\authorinfo[A. Khovanskii]{
    Department of Mathematics, University of Toronto, Canada}{%
    askold@math.utoronto.ca
    }

\abstract{%
The notion of center of mass, which is very useful in kinematics, proves to be very handy in geometry (see \cite{1}--\cite{2}). Countless applications  of center of mass  to
geometry  go back  to Archimedes. Unfortunately, the center of mass cannot be defined for  sets  whose total mass equals zero. In the paper we improve
this disadvantage  and assign  to an  $n$-dimensional affine space $L$ over
any field  $\bold k$ the $(n+1)$-dimensional vector space  $\widehat M(L)$ over $\bold k$ of weighty
 points and mass dipoles in $L$. In this space the sum of weighted points
whose total mass  $\lambda\neq 0$ is equal to the center of mass  of these
points, equipped with  mass $\lambda$. We present several interpretations
of the space $\widehat M(L)$ and a couple  of its applications to geometry. The
paper is self-contained and is accessible for undergraduate students.
    }


\VOLUME{31}
\YEAR{2023}
\NUMBER{3}
\firstpage{49}
\DOI{https://doi.org/10.46298/cm.11527}
\begin{document}



\section{Introduction}\label{sec1}

\noindent{\bf 1.1.} I first met Sergei Duzhin  a long time ago: Sergei was an active member of the famous Arnold's mathematical seminar in Moscow which was an important part of my life.

Sergei was an attractive and original person. He was  easily  learning foreign languages, and while visiting universities in different countries he  was starting to lecture in the language of the host country. He took an accordion with him on trips, often sang his favorite songs, and his friends
and colleagues  sang along with him. Sergei organized interesting
mathematical seminars, first in Pereslavl-Zalessky, and then in St.\ Petersburg, which attracted many visitors. He was very friendly, and    like many others I enjoyed his company.

He had a wide range of interests, for example, he took great pleasure in
studying the history of St.\ Petersburg. He was interested in very different areas of mathematics, and I would be glad to discuss the center of mass technique with him.
\bigskip

\noindent{\bf 1.2.} Center of mass of a system of points which are
 equipped with real non necessarily positive masses can be defined on any real Euclidean, spherical or hyperbolic space and on some other spaces (see \cite{3}). Such generalization of the classical center of mass technique does not come for free. For example,  in  the ``calculus of masses'' on a sphere centered at the origin one identifies a point $A$ equipped with a mass $m$  with the opposite point $-A$ equipped with the mass $-m$.

In this paper we develop the center of mass technique for an affine space over an arbitrary field. We generalize the classical approach and include into consideration the case when the total mass of a system is equal to zero. In that case instead of its center of mass  equipped with the total mass of the system, one associates to the system its mass dipole (see below), which
could be considered as a free vector. This procedure has an interpretation in projective geometry. Consider the natural projective compactification of the affine space. To a free vector one associates the point at infinity hyperplane at which the free vector is pointing.
The center  of mass in the case under consideration can be interpreted as this point at the infinite hyperplane equipped with the ``mass'' equals the free vector.
In this paper we will not discuss this (very useful) interpretation and will restrict yourself by affine geometry not touching projective geometry.
\bigskip

\noindent{\bf 1.3.} Applications of the center of mass technique to
geometry go back  to Archimedes. They are based on the existence of the center
of mass of any {\sl weighty set of points} (i.e., of any finite set of
points equipped with masses  whose total mass is not equal to zero).
Center of mass  satisfies some nice properties (see Axioms 1 and 2 in
Section \ref{subsec2.2}) which suggest its applications to geometry.

 As Archimedes, one can make geometrical discoveries  heuristically,
believing in the existence of the center of mass (see Section
\ref{sec2}). Of course, it is not hard to rigorously justify the center
of mass technique (see Theorems \ref{theorem 1.33}--\ref{induction}
below).
 \bigskip

 \noindent{\bf 1.4.} One tempted to look for  a commutative group of
weighty  points (i.e., of points equipped with non zero masses) in an
affine space $L$ such that the sum  of weighty  points in the group is equal to the center
of mass  of these points equipped with their  total  mass.

Such a group cannot be defined: one cannot sum a set of   weighty points,
whose total mass is equal to zero, since the center of mass of such a set
does not exist.

One can improve this situation and  define a vector space $\widehat M(L)$,
whose elements are weighty points and {\sl mass dipoles} in $L$, defined
up to a shift.

Instead of the center of  mass, one can assign to  any {\sl weightless
set}, i.e., to  a set of points whose total mass is equal to zero,   its
 mass dipole (see below) defined up to a shift.

 A {\sl mass dipole} $\{-A,B\}$ in $L$
is an ordered  pair of points $A,B$ equipped with  masses $-1,1$,
correspondingly. Two mass dipoles $\{-A,B\}$ and $\{-C,D\}$ are {\sl
equivalent}, if the oriented segments $AB$ and $CD$ are parallel, equal
and having the same direction (in the other words, if the ordered pairs of
points $(A,B)$ and $(C,D)$ are equal up to a shift).

In the  additive group of the space $\widehat M(L)$ the  sum   of weighty
points, whose  total mass  is $\lambda\neq 0$, equals  the  center of mass
of the set of these points, equipped with   mass $\lambda$.

The space $\widehat M(L)$ has many geometrical applications besides classical
applications (see [1]--[2]) of the center of mass. To show how it works we
present an example of such application in  Section~\ref{sec8}.  One can start reading the paper with this section, looking for needed
definitions in previous sections.
\bigskip

\noindent{\bf 1.5.}  In this paper, we  introduce  several interpretations of the vector space~$\widehat M(L)$ over  $\bold k$  of weighty points and mass dipoles in an affine space $L$ over $\bold k$. Let us briefly discuss these interpretations.

With the affine space  $L$ one can associate the vector space $M(L)$ of
{\sl moment-like} affine maps $P:L\to \overline L$ of $L$ to the space
$\overline L$ of free vectors in $L$  whose {\sl linear part} (which is a map  of  the space $\overline L$ to itself) is proportional to the
identity map. The space $M(L)$ contains a  subspace $M_0(L)$ of
codimension one which consists  of  constant maps  $P:L\to \overline L$ whose linear parts are equal to zero.

The space $\widehat M(L)$ is isomorphic to the space $M(L)$. We prove basic properties of the space $\widehat M(L)$  using this isomorphism.

The space $\widehat M(L)$ can be defined as the factor-space $D(L)/DM(L)$ of the infinite dimensional vector space $D(L)$ of all  weighted sets in $L$,
i.e., all finite sets of points in $L$ equipped with masses $\lambda_i\in
\bold k$,   by its subspace $DM(L)$, consisting  of {\sl null sets} in
$L$, whose total masses and mass dipoles are equal to zero (see Section
\ref{subsec5.2}). The isomorphism $\widehat M(L)\sim D(L)/DM(L)$ allows to assign to any affine map $F:L_1\to L_2$ the corresponding linear map
$F_*:\widehat M(L_1)\to \widehat M(L_2)$.

The space $\widehat M(L)$ consists of weighty  points and mass dipoles in $L$ defined up to a shift. One can define vector space operations  (either  an addition of two vectors or  a multiplication of a vector by an element of the field $\bold k$) on the space $\hat
M(L)$ by listing rules which allow to  perform these operations (see
Section \ref{subsec7.1}). Such presentation of the space $\widehat M(L)$ is
the most convenient for geometrical applications.

If an affine space $L$ is  an affine hyperplane, not passing through the
origin in a vector space $\bold L$, then the space $\widehat M(L)$ can be
identified with the ambient vector space $\bold L$. Such identification
provides the most visual  interpretation  of the space $\widehat M(L)$. It
works only if $L$ is an affine hyperplane in $\bold L$ not passing through
its origin.

One can  canonically represent any affine space $L$ as the {\sl
characteristic hyperplane} (see  Section \ref{subsec 10.3}) not passing
through the origin in the space $P_1^*(L)$, dual to the space $P_1(L)$ of
polynomials on $L$ whose degree is $\leq 1$ (see  Section \ref{subsec
10.3}). This representation implies that the space $\widehat M(L)$ is
canonically isomorphic to the space $P_1^*(L)$.

Let $L_B$ be a vector space  of all polynomials  on $L$ of degree $\leq
2$,  whose homogeneous parts of degree 2 are proportional to a
non-degenerate quadratic form $Q_B$  (whose symmetric  bilinear form is a fixed form $B$).

One can show that the space $\widehat M(L)$ is isomorphic to the space of
differentials of  polynomials from the space $L_B$. This observation
implies that the critical points of  polynomials from the space $L_B$ obey
the same laws as the centers of mass of weighty points in the affine space
$L$ (see Section~\ref{subsec10.4}).
\bigskip

\noindent{\bf 1.6.} A few words about the organization of this paper.

Section \ref{sec2} contains an introduction of the classical center of mass technique. This  technique has the following disadvantage:
center of mass cannot be defined for  sets whose total mass is equal to zero.
In Section \ref{sec3}, we discuss how to improve this method. Next sections contain a detailed presentation of such improvement.

In Section \ref{sec4}, we recall definitions and basic properties of affine spaces over arbitrary fields, and of affine maps. Analogous  material can be found in many places,  see, for example \cite{4}.

Sections \ref{sec5}--\ref{sec7} contain the central definitions and results of the paper: here we discuss  weighted sets of points in affine spaces, their moments and moments maps, moment-like maps and the vector space $\widehat M(L)$ of weighty points and mass dipoles in an affine space $L$.

 In Section \ref{sec8}, we prove classical theorems on three altitudes in a triangle and on the Euler line in a triangle, applying mass dipoles and centers of mass.

In Section \ref{sec9}, we discuss relations between affine geometry of a hyperplane $L$ in a vector space $\bold L$, not passing through its origin, and geometry of the vector space $\bold L$. We also show that any affine space can be canonically embedded into a vector space as
an affine hyperplane not passing through the origin.

In Section~\ref{sec10},  we show first that an affine map $F:L_1\to L_2$ induces
 a linear map $F_*:\widehat M(L_1)\to \widehat M(L_2)$. Then we present several
interpretations of the space $\widehat M(L)$.
\bigskip

\noindent{\bf 1.6.} Vladlen Timorin made many valuable suggestions which allowed me to improve the
exposition. He also edited my English. My wife Tatiana Belokrinitskaya helped me
writing the paper. In particular, she typed and edited it. I am grateful to both of
them: without their help, I would not be able to complete this project.

{This work was partially supported by the Canadian Grant No.\ 156833-17. }

\section{Heuristic  Applications of Centers of Mass in Geometry}\label{sec2}

A set of {\sl weighted points} $\widetilde A=\{(A_i, \lambda_i)\}$ in a real $n$-dimensional space $L$ is a finite set of points $A_i\in L$ equipped with (positive or negative) numbers $\lambda_i$ which are called ``masses'' of the corresponding points. The {\sl total mass} $\lambda=\lambda_T(\widetilde A)$ of the set $\widetilde A$  is the sum  $\sum \lambda_i$ of masses of all points in the set. The set is {\sl weighty} if its total mass is not equal to zero, and is {\sl weightless} if its total mass is equal to zero.

According to  kinematics,  one can assign  to each weighty set $\widetilde A$  a point which is called {\sl the center of mass of the set $\widetilde A$}.  The map, that sends each weighty set to its center of mass, satisfies some nice properties called   Axioms 1 and 2 (see below).

\subsection{Intuitive Meaning of  the Center of Mass}\label{subsec2.1}

Let us discuss an intuitive meaning of the  center of mass of a weighty set $\widetilde A$ in a plane~$L$.

One possible way of understanding what the center of mass means is to imagine the plane $L$ as a big flat weightless tabletop. Imagine also that to each point  $A_i$ equipped with a mass $\lambda_i$ a  force proportional to $\lambda_i$ is applied that pulls the tabletop  down if $\lambda_i$ is positive, or  pulls it up if $\lambda_i$ is negative.

It is known  from numerous experiments that one can support the tabletop on just one leg; the position of the leg depends on the weighty set $\widetilde A$. This point is called  the {\sl center of mass} of $\widetilde A$.

Let us see what the center of mass is for a system of two  positive point masses $m_1$ and $m_2$ located at points $A$ and $B$.

Experiments show that if $m_1$ and $m_2$ are positive, the center of these masses is located at the  point $C$ on the line segment $AB$ such that $$m_1\cdot AC=m_2\cdot CB.$$

Here $AC$, $CB$ are the oriented lengths of the corresponding segments. Since the point $C$ is located between the points $A$, $B$, the oriented lengths have the same sign and can be taken positive.

One can consider negative masses as well. A point with a negative mass could be considered as a point of application of a force proportional to its mass and acting not down but up.

The center of a  weighted couple of points $A$ and $B$, equipped with  (not necessarily positive) masses $m_1$ and $m_2$, is defined by the same formula, but one should allow the point $C$ to be located on the line $AB$, not necessarily in the segment $AB$.

\begin{figure}[htbp]
\begin{center}
  \includegraphics[scale=1.5]{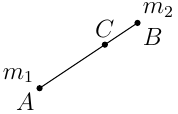}
\end{center}
\label{fig:CenterOfMassDefinition}
 {\caption{If the mass at $B$ is twice as big as the mass at $A$, then the center of mass $C$ will be two times closer to $B$ than to $A$}}
\end{figure}

\begin{figure}[htbp]
\begin{center}
  \includegraphics[scale=0.8]{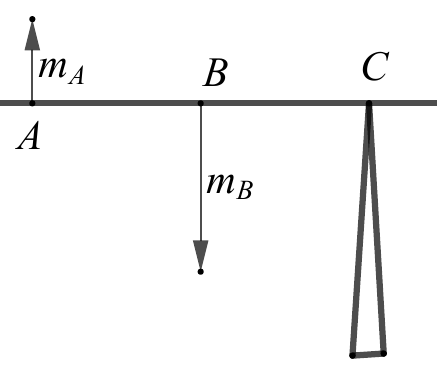}
\end{center}
\label{fig:negativemass}
\caption{If $AC=2BC$, then equilibrium holds if $m_B=-2m_A$}
\end{figure}

The identity defining $C$ can be rewritten more symmetrically:
$$m_1\cdot CA +m_2\cdot CB=0.$$

If points $A$ and $B$ are equipped with masses $-m$ and $m$ whose sum is equal to zero, then their center of mass is not defined. If $m$ is the unit mass, $m=1$, the corresponding pair $\{-A,B\}$ is called  mass dipole.

Two  mass dipoles $\{-A,B\}$ and $\{-C,D\}$ are equivalent if the ordered pairs  of points $(A,B)$ and $(C,D)$ are equal up to a shift.

For the sake of equilibrium  of the tabletop, supported at the pivot point $O$,  the effect of putting a mass dipole $\{-A,B\}$ and putting a mass dipole $\{-C,D\}$  on the tabletop, are equal, if the  mass dipoles are equivalent.

Mass dipoles together with weighty points play a key role in the extension of the center of mass technique presented in the paper.

The center of mass of a weighty set containing  more than two points can be found inductively using Axiom 2 (see below).

\begin{figure}[htbp]
\begin{center}
 \includegraphics[scale=0.5]{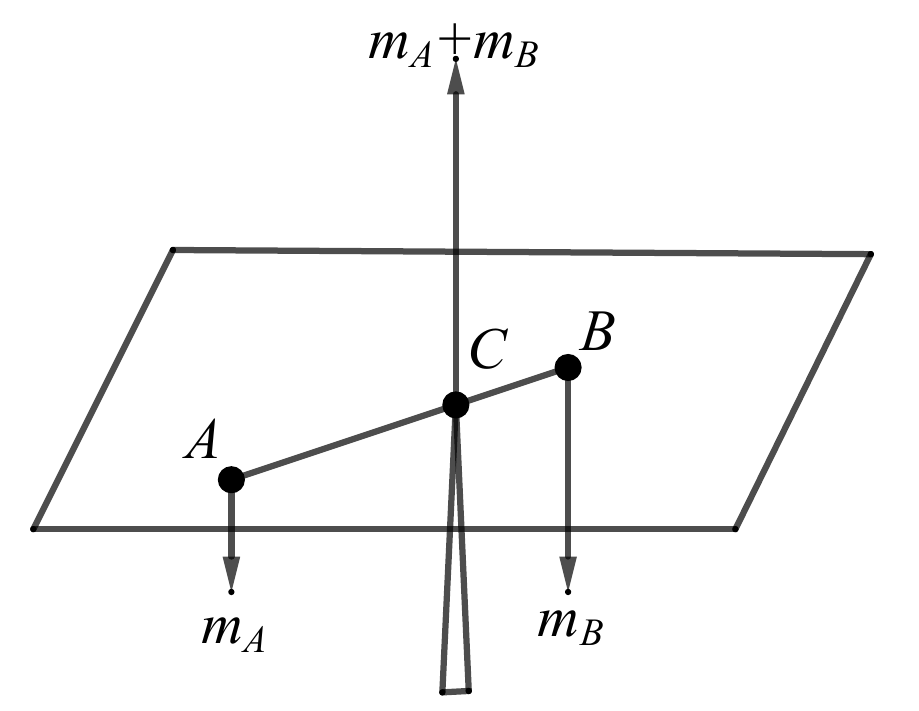}
\end{center}
\caption{Force proportional to $m_1+m_2$ must be placed at $C$ to keep the table in equilibrium}
\label{fig:CenterOfMassIntuition2}
\end{figure}

\subsection{Axiomatic Definition of the Center of Mass}\label{subsec2.2}

According to kinematics, centers of mass of  weighty sets in a real $n$-dimensional affine  space $L$ must  satisfy the following axioms.
\medskip

\noindent{\bf Axiom 1}. The center of mass of two points $A$ and $B$ equipped with masses $m_1$ and $m_2$ with nonzero sum $m_1+m_2$ is the unique point $C$ on the line $AB$ such that $$m_1\cdot AC=m_2\cdot CB\,\, \mbox{or }\,\, m_1\cdot CA + m_2\cdot CB= 0.$$

\noindent{\bf Axiom 2}.  Let $\widetilde B$ be a weighty subset of a weighty set $\widetilde A$. Let $O_1$ be the center of mass of $\widetilde B$ and let  $m_1$ be the total mass of $B$.
\medskip
Then the center of mass of $\widetilde A$ coincides with the center of mass of the set obtained from $\widetilde A$ by replacing the subset $\widetilde B$ by the point $O_1$ equipped with  mass $m_1$.

\begin{theorem}
\label{theorem 1.33}
There is a unique way of assigning the center of mass to any weighty set of points in a real $n$-dimensional affine space  $L$ so  that   Axioms 1 and 2 hold.
\end{theorem}

We will not prove  Theorem \ref{theorem 1.33} now. Its classical proof can be described in the following way (see proof of Theorem \ref{theorem 1.33}, presented right after Lemma \ref{two weighty points}).

One can define the center of mass by an explicit formula  and check that the center of mass thus defined satisfies Axioms 1, 2. Uniqueness of the center of mass (see Theorem \ref{induction} below) implies that the explicit formula provides the unique possible definition of the center of mass and proves its existence.

\begin{theorem}\label{induction} Assuming that there is a  way of assigning a center of mass to any weighty set of points in a real $n$-dimensional affine space  $L$ so that   Axioms 1 and 2 hold,  one  can compute  the center of mass of any weighty set (in many different ways) using an inductive procedure presented in the proof of Theorem.
\end{theorem}

\begin{remark} Theorem does not imply the existence of a center of mass satisfying Axioms 1, 2: different ways of computation of the center of mass could provide different answers.
\end{remark}

\begin{remark} Theorems \ref{theorem 1.33},  \ref {induction} hold for affine spaces over arbitrary field $\bold k$ whose characteristic is not equal to two.These theorems are not applicable to  affine spaces over  fields of characteristic two (the proof of Lemma \ref{lemma2.1} implicitly  uses  division by two and it does not work over fields of characteristic two). Our extension of the center of mass technique works for affine spaces over arbitrary fields.
\end{remark}

\begin{lemma}\label{lemma2.1} If a weighty set $\widetilde A$ contains at least two points, then it contains a weighty subset of two points.
\end{lemma}

\begin{proof} If all subsets of $\widetilde A$ containing  two points are weightless, then weights of all points are equal up to sign. A triple of such weighty points contains a pair of weighty  points equipped with equal masses. This pair of points is a weighty subset of $\widetilde A$.
\end{proof}

\begin{proof}[Proof of Theorem \ref{induction}] Using Axioms 1, 2  one can compute the center of mass of any weighty set of points. Indeed:

\begin{itemize}

    \item  Axiom 1 allows  to compute the  center of mass for any weighty set containing two points;

   \item by Lemma \ref{lemma2.1}, a weighty set containing at least  two points, contains a weighty pair of  points. By Axiom 2, one can replace this pair of points by its center of mass equipped with  its total mass;

  \item  the above properties  allow to find the  center of mass of a  weighty set  recursively,  by reducing it to computation of the center of mass of weighty sets, containing smaller number of points;
\end{itemize}
\end{proof}

Theorems \ref{theorem 1.33}, \ref{induction}  have countless applications in geometry: Theorem \ref{theorem 1.33} implies that {\sl  all the different ways of finding  the center of mass of a given weighty set, suggested  by Theorem \ref{induction}, have to give  the same answer}.

Classical applications of  centers of mass in geometry are based on the above statement.

 Let us present the simplest classical application of centers of mass in geometry which goes back to
 Archimedes.

\subsection{Three Medians in a Triangle (Archimedes)}\label{Archimedes}

Let us prove the theorem on three medians in a triangle using the center of mass technique. A similar  proof was discovered by Archimedes who invented the center of mass technique in geometry.

Suppose we start with a triangle $ABC$ and three unit mass at its vertices. How can we find the center of mass $O$ of the resulting system?

One way is to combine masses at $A$ and $B$ first. This will give us mass 2 placed at the midpoint of $AB$.

Now,  combine the resulting mass with the unit mass at $C$. By doing so, we obtain that the center of mass $O$ lies on the segment connecting the vertex $C$ to the midpoint of $AB$ and divides it in the proportion $2:1$ (the part of it adjacent to the vertex being the longer one).

\begin{figure}[htbp]
\begin{center}
 \includegraphics[scale=1.1]{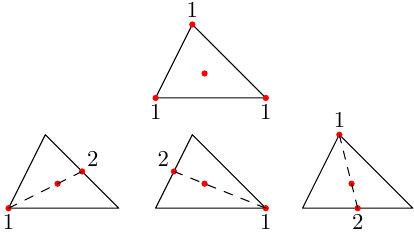}
\end{center}
\caption{Center of 3 unit masses and 3 medians}
\label{fig:MediansUsingCenterOfMass}
\end{figure}

But we could proceed differently: first, combine the masses at $A$ and $C$, and then combine the result with the mass at $B$. In this way we obtain that the center of mass $O$  of the system lies on the median from vertex $B$ and divides it in the ratio $2:1$.

By combining the masses in the third possible order ($B$ and $C$ first and adding $A$ afterwards), we see that it also lies on the median from the vertex $A$ and divides it in the ratio $2:1$.

Thus, the center of mass $O$ lies on all three medians of triangle $ABC$ and divides them in the ratio $2:1$. We proved that
all three medians of the triangle $ABC$ pass through a point $O$ which  divides them in the ratio $2:1$.

\section{How to Improve the Center of Mass Technique?}\label{sec3}

The classical center of mass technique has the following disadvantage: the center of mass is not defined for weightless sets.

One is tempted to look for a commutative group of weighty  points in $L$ with the following addition: the sum  of weighty  points is equal to the   center of mass of the set of these points, equipped with its total  mass.

Unfortunately, such a group does not exist: the sum of  weighted points whose total mass is equal to zero, is not defined.

\subsection{The Group $\widehat M(L)$ of Weighty Points and Mass Dipoles}\label{subsec3.1}

It is possible to improve the  center of mass technique and to define a commutative group $\widehat M(L)$ whose elements are weighty points in $L$  and  mass dipoles in $L$, defined up to a shift.

Instead of the center of  mass, one can assign to any weightless set  {\sl its  mass dipole}, defined up to a shift (one can identify a mass dipole $\{-O,B\}$, defined up to a shift, with the free vector, represented by an ordered pair of points $(O,B)$, defined up to a shift).

Addition in the group $\widehat M(L)$ satisfies the following condition: the  sum   of   points, contained in a set $\widetilde A$, whose  total mass  is $\lambda\neq 0$, equals to the  center of mass of $\widetilde A$, equipped with the  mass $\lambda$.

Addition in the  commutative   group $\widehat M(L)$ could be defined axiomatically. Instead of one axiom, describing the  center of  mass of a weighty pair of  points, one has to introduce  four axioms, describing the following four different types of addition in  $\widehat M(L)$:

\begin{itemize}

\item addition of two weighty points whose total mass is not equal to zero;

\item  addition of two weighty points whose total mass is  equal to zero;

\item  addition of a weighty point  and a mass dipole defined up to a shift;

\item  addition of two mass dipoles defined up to a shift.
\end{itemize}

Numerous geometric applications of the group $\widehat M(L)$ are based on its associative property:

\noindent{\sl addition of elements of the group $\widehat M(L)$ performed in different  orders gives the same  answer.}

Along with classical  applications dealing only with the centers of mass of weighty sets, there are many natural applications, that use weightless sets and their mass dipoles as well. To show how it works, we describe one simple geometric application of the group $\widehat M(L)$ (see Section \ref{sec8}).

Below we define the  group $\widehat M(L)$ for any affine space $L$ over an arbitrary field $\bold k$. Moreover, we show that the group $\widehat M(L)$ admits a structure of a vector space over the field $\bold k$.

The vector space $\widehat M(L)$ can be naturally  identified with the space of {\sl moment-like maps $P:L\to \overline L$ } of the affine space $L$ to the space of free vectors $\overline L$ in $L$.  We  deduce properties of the group $\widehat M(L)$ from the corresponding  properties of the additive group of  $M(L)$.

The definition   of moment-like maps is motivated by kinematics. In the next section, we discuss a heuristic    equilibrium criterion which suggests
an explicit formula for the center of mass of a  weighted set of points. It allows to justify the classical center of mass technique. It also leads to definitions of moment maps of weighted sets, of moment-like maps  and  of group $\widehat M(L)$.

 \subsection{Moments of Weighted Sets about Pivot Points and  Moment  Maps}\label{subsec3.2}

 Imagine a weightless tabletop $L$ with  a weighted set of points $\widetilde A$ on it. Assume that the tabletop is supported by one leg  at a point $O$.

\begin{definition} The {\sl moment $P(O)=\Delta_{\widetilde A}(O)$   of a weighted set $\widetilde A= \{(A_i,\lambda_i)\}$ about a pivot point $O$}  is a linear combination of  oriented  segments  $OA_i$ (which we consider as  free vectors in  $L$) with  coefficients $\lambda_i$.
\begin{equation}\label{equation: formula for moment 1}
P(O)=\Delta_{\widetilde A}(O)=\sum \lambda_i OA_i.
\end{equation}
\end{definition}

Kinematics suggests a simple criterion of stability of the tabletop with the set $\widetilde A$ on it, supported at a point $O$.
\medskip

\noindent{\bf Kinematical Criterion}.
{\it  The tabletop is in equilibrium if and only if  the moment $P(O)=\Delta_{\widetilde A}(O)$ of  $\widetilde A$  about the pivot point $O$  equals  zero.}
\medskip

Only the moment of the set $\widetilde A$ about the pivot point $O$ is important for the  equilibrium of a tabletop supported at  $O$:
for the sake of equilibrium  of the tabletop supported at  $O$,  the effects of putting a weighted set $\widetilde A$ and putting a weighted set $\widetilde B$  on the tabletop are equal  if the moments of $\widetilde A$ and $\widetilde B$ about the  point $O$ are equal.

\begin{definition} The {\sl moment map $P:L\to \overline L$ corresponding   to a weighted set $\widetilde A$} is  the map  $P=\Delta_{\widetilde A}$, whose value $P(O)$ at  a point $O\in L$ is equal to the moment $\Delta_{\widetilde A} (O)$ of the set  $\widetilde A$ about the pivot point $O$.
\end{definition}

\begin{theorem}[Change of  pivot point]\label{change of pivot} Let $P=\Delta_{\widetilde A}$ be the  moment map,  corresponding to a weighted set $\widetilde A$
whose total mass is equal to $\lambda=\lambda_T(\widetilde A)$.
Then, for any two points $O_1, O_2\in L$, the map $P$  satisfies the following  characteristic relation:
\begin{equation}\label{pivot is affine}
P(O_1)-P(O_2)=-\lambda O_2O_1.
\end{equation}
\end{theorem}

\begin{proof} By (\ref{equation: formula for moment 1}), we have: $$ P(O_1)=\sum \lambda_i  O_1A_i \quad {\text {and}}\quad   P(O_2)=\sum \lambda_i O_2 A_i.$$

Subtracting  these identities, we obtain
$$P(O_1)-P(O_2)=-\lambda O_2O_1.$$
\end{proof}

\begin{definition} A {\sl moment-like map} $P:L\to \overline L$ is a map  of the affine space $L$ to the vector space $\overline L$ of free vectors on $L$ that satisfies the relation  (\ref{pivot is affine})  for some parameter   $\lambda\in \bold k$.
The parameter $\lambda$, which depends on the map $P$, is called the  {\sl total mass} of the map $P\in M(L)$.
\end{definition}

It is easy to see that the set $M(L)$ of all moment-like maps  has a natural structure of a vector space, and the total mass $\lambda$ is a linear function on $M(L)$.

Applying Theorem \ref{ceter of mass is unique}, one can check  that, if the total mass $\lambda$ of a map $P\in M(L)$ does not equal to  zero,  then $P$ vanishes at a unique point $O$ called the {\sl center of mass of the map $P\in M(L)$}. Such a map is uniquely determined by its center of mass $O$ and its total mass $\lambda$ (see Theorem \ref{center of mass and total mass}).

Moreover, the map  $P$  is equal to the moment map $P=\Delta (\widetilde A)$  of  the weighty set  $\widetilde A$   consisting of one point $O$ equipped with the mass $\lambda$.

Equation (\ref{pivot is affine}) implies that, if the total mass $\lambda$ of a map $P\in M(L)$ equals zero, then $P$ is a constant map, i.e., $P\equiv v$, where $v \in \overline L$ is a free vector. The map $P$ is  determined  by the free vector $v$.

\begin{lemma}\label{moment map of mass dipole} The moment map $P=\Delta_{\widetilde A}$  of a mass dipole $\widetilde A=\{-O,B\}$ is a constant map  $P\equiv v$, where $v$ is the free vector  represented by an ordered pair of points $(O, B)$.
\end{lemma}

\begin{proof} The moment of $\widetilde A$ about a pivot point $O_1$ is equal to the difference $O_1B-O_1O$ of oriented segments $O_1B$ and $O_1O$, considered as  free vectors. This difference is equal to the oriented segment $OB$  which  can be considered as the free vector $v$, represented  by an ordered pair of points $(O,B)$.
 \end{proof}

Let us  introduce the following identifications:
\begin{itemize}

\item a moment-like map $P$ with  total mass $\lambda\ne 0$ is identified with the weighty point $\{(O,\lambda)\}$, where $O$ is the center of mass of $P$;

\item a constant moment-like map $P\equiv v$ is identified with  a mass dipole $\{-O,B\}$, defined up to a shift,  the ordered pair of points $(O,B)$ represents the free vector $v$.
\end{itemize}

By this identification, we obtain a vector space $\widehat M(L)$ whose elements are weighty points and mass dipoles, defined up to a shift. The additive  group of the space $\widehat M(L)$ can be defined by the axioms which define the addition in the set $\widehat M(L)$.

Let us implement this program in more details.

\section{Affine Spaces  and Affine Maps}\label{sec4}

Below, we recall  that Euclidean spaces,  vector spaces over an arbitrary field  $\bold k$ and shifted vector subspaces in such spaces can be considered as affine spaces.

\subsection{Vector Space of Free Vectors in  a Euclidean Space}\label{subsec4.1}

Classical Euclidean geometry, in particular,  deals with points, lines and planes embedded in the three dimensional space. These objects are {\sl  Euclidean spaces}  of dimensions zero, one, two and three.

A Euclidean space $E$ does not have a structure  of a real vector  space: an operation of multiplication of a point $A\in  E$ by a real number $\lambda\in \mathbb R$  and an operation of addition of two points $A, B\in E$ are not defined.

Nevertheless, to any Euclidean space $E$ one can assign   a real vector space $\overline E$ of {\sl  free vectors} in $E$ and define an action of the  additive group  of free vectors  on  the space  $E$.

Let us define free vectors in a Euclidean space $E$.

\begin{definition} Two ordered pairs of points $(A,B)$ and $(C,D)$  from the space $E$ are {\sl equivalent}  $(A,B)\sim (C,D)$,  or {\sl are equal up to a shift},   if   the oriented segments $AB$ and $CD$   are parallel, equal in length and point in the same direction.
\end{definition}

\begin{definition}  A free vector in $E$  is an ordered pair of points $(A,B)$ in $E$, defined up to the equivalence $\sim$.
\end{definition}

One can multiply an  ordered pair of points $(O,A)$  by any real number $\lambda$.  By definition, $\lambda (O,A)$ is an ordered couple of points $(O,B)$
such that the points $O,A,B$ belong to the same line, and ratio $OB:OA$
of oriented length of the segments $OB$ and $OA$ equals $\lambda$.

It is easy to check that if $(A_1,B_1)\sim (A_2,B_3)$, then $\lambda(A_1,B_1)\sim \lambda (A_2,B_3)$.

\begin{definition} Let $(A,B)$ be an ordered pair of points that represents a free vector $v\in \overline E$. Then, by definition, $\lambda v\in \overline L$ is the free vector represented by the ordered pair $\lambda (A,B)$.
 \end{definition}

 One can add any two ordered pairs of points  $(O,A)$ and $(O,B)$.  By definition, $(O,A)+(O,B)$    is an ordered pair of points $(O,C)$ such that the oriented segments $AC$ and $OB$ are parallel, equal to each other and point to the same direction.

 Easy to check that if $(A_1,B_1)\sim (A_2,B_2)$ and $(A_1,C_1)\sim (A_2,C_2)$, then
$(A_1,B_1) + (A_1,C_1)\sim (A_2,B_2) + (A_2,C_2)$.

\begin{definition} Let $(A,B)$, $( A,C)$ be  ordered pair of points which represent  free vectors $V_1,V_2 \in \overline E$. Then, by definition, $V_1+V_2\in \overline L$ is the free vector represented by the oriented pair $(A,B) +(A,C)$.
 \end{definition}

 One can verify that if  $E$ in an $n$-dimensional Euclidean space   ($n=0,1,2,3$), then the space $\overline E$ of free vectors in $E$ is  an $n$-dimensional real  vector space.

There is a natural operation of addition for a {\sl free vector} and a {\sl point in the space}  $E$. The result of such addition is  a point in the space $E$.

One can add a vector to a point.

\begin{definition}  By definition, a point $P \in E$ is the sum $O+v$ of a point $O\in E$ and a free vector $v\in \overline E$ if the  ordered pair of points $(O,P)$ represents the free vector    $v\in \overline E$.
\end{definition}

It is easy to check the following lemma.

\begin{lemma}\label{lemma4.1} Let $O$ be any point in $E$. Then
\begin{enumerate}

\item the identity  $O+v=O$  holds if and only if  $v$ is   the zero vector in the space of free vectors $\overline E$;

\item for $v_1,v_2\in \overline E$  the identity  $O+(v_1+v_2)=(O+v_1)+v_2$ holds.

 \item for any point $P\in E$,  there is a unique free vector $v\in \overline E$ such that $P=O+v$.
\end{enumerate}
\end{lemma}

A free vector $v$ defines the shift $Sh_v$ of the space $E$ that maps a point $O\in E$ to the point $Sh_v(O)=  O+v$.

Lemma \ref{lemma4.1}  implies that the shifts $Sh_v$  constitute the action of the additive group of free vectors on the space $E$. Moreover, this action is transitive and free.

\begin{remark} Euclidean geometry equips the vector  space $\overline E$ of free vectors in $E$ with the inner product $\langle x,y\rangle$ of vectors $x,y\in \overline E$.  The inner product is a symmetric bilinear form in the vector space $\overline E$ that is positively defined, i.e., $\langle x,x\rangle\geq 0$ and $\langle x,x\rangle=0$ only if $x=0$. The inner product allows  defining distances  between points and angles between lines in a Euclidean space $E$.
\end{remark}

\subsection{Vector Spaces  as  Affine Spaces}\label{subsec4.2}

Let $\bold L$ be a vector space  over an arbitrary field $\bold k$. In this section, we  define the structure of an affine space on $\bold L$, i.e., we  define the vector  space $\overline{\bold L}$ of free vectors in $\bold L$ and the action of the additive group of free vectors on the space $\bold L$.

\begin{definition} A {\sl shift  of the space $\bold L$ by a vector} $v\in \bold L$ is the map $Sh_v:\bold L \to  \bold L$ which sends a point $O\in \bold L$  to the point $O+v$.
\end{definition}

The additive group of the vector space $\bold L$ acts on the space $\bold L$ by shifts: a vector $v\in \bold k$ sends a point $O\in \bold L$ to the point $Sh_v(O)$.

\begin{definition} A free vector in $\bold L$ is an ordered pair of points $(A,B)$ defined up to a shift (i.e., for any  $v \in \bold L $ the ordered  pair of  points $(A+v, B+v)$ defines the same free vector as the ordered pair of points $(A,B)$).
\end{definition}

One can identify  a free vector  $f\in \overline{\bold L}$, represented by a pair $(A,B)$, with  the vector $v=B-A\in \bold L$.  By definition, the vector $v$ is well defined, i.e., is independent of a choice of a pair $(A,B)$, that represents the free vector $f$.

\begin{definition} By definition, a {\sl vector space operation on the space $\overline{\bold L}$ of free vectors} is induced from the corresponding  vector space operation on the vector  space $\bold L$ under the identification  of   $\overline{\bold L}$ with the vector space $\bold L$.
\end{definition}

Thus, we defined  the vector space $\overline{\bold L}$ of free vectors in $\bold L$.  The  action of the additive group of  $\overline{\bold L}$  on $\bold L$ is defined as follows: a free vector $(A,B)\in \overline{\bold L}$ sends a point $O\in \bold L$ to the point $Sh_{B-A}(O)=O+B-A\in \bold L$.

\subsection{Shifted  Vector Subspaces  as  Affine Spaces}\label{subsec4.3}

An {\sl affine subspace} $L\subset\bold L$  is a {\sl  vector subspace} $\widetilde L\subset \bold L$ shifted by a vector $O\in \bold L$. In general, if $O$ does not belong to $\widetilde L$, the space   $L\subset \bold L $ is not a vector space. Instead, it can be naturally  equipped with  a  structure of an affine space, i.e.,  with  a vector space  $\overline L$ of free vectors on $L$ and with an  action of the additive group of free vectors $\overline L$ on the space $L$.

We  define an affine structure on $L$ using vector space operations in the ambient space $\bold L$.

\begin{definition} A set $L\subset \bold L$ is a {\sl shifted subspace} $\widetilde L\subset \bold L$ by a vector $O\in \bold L$ if $L=Sh_O(\widetilde L)$, i.e., if  each point $A\in L$ is representable in the form $A=\widetilde A+O$, where   $\widetilde A\in \widetilde L$. (In other words, the set  $L$ is a coset of the  additive group of the vector subspace  $\widetilde L$, containing the point $O$.)
\end{definition}

\begin{definition} Two ordered pairs of  points $(A,B)$ and $(C,D)$ in an affine space  $L$ are {\sl equivalent} if $B-A=D-C$. (In other words, the pairs  are equivalent if they define the same free vector in the ambient space $\bold L$.)

 A free vector in an affine space  $L$ is an equivalence class of  ordered pairs of points in $L$.
\end{definition}

There is a natural embedding of the set $\overline L$ to the vector space $\overline{\bold L}$ which sends a free vector in $\overline L$, represented by an ordered pair of points $(A,B)$ from $L$, to the free vector in $\overline{\bold L}$ represented by the same ordered pair of points $(A,B)$.

\begin{lemma} Each free vector
 $(A,B)\in \overline L$ as a free vector in $\overline{\bold L}$ has a unique representative of the form $(\bold O, C)$, where $\bold O$ is the origin of the vector space $\bold L$, and $C$ belongs to the vector space $\widetilde L$.
\end{lemma}

\begin{proof} Indeed, a free vector  $(A,B)$ in $\overline{\bold L}$ is equivalent  to a unique ordered pair of points $(\bold O,B-A)$ of the needed form.
\end{proof}

One can identify  a free vector  $f\in \overline L$, represented by a pair $(A,B)$ with  the vector $v=B-A\in \widetilde L$.  The vector $v$ is well defined, i.e., is independent of a choice of a pair $(A,B)$ which represent the free vector $f$.

\begin{definition} By definition, a {\sl vector space operation}  on the space $\overline L$ of free vectors in $L$ is induces from  the corresponding  vector space operation on the  space $\widetilde  L$ under the identification  of   $\overline L$ with the vector space $\widetilde L$.
\end{definition}

Thus, we defined  the vector space $\overline L$.  The action of the additive group of  $\overline  L$  on $L$ is defined as follows: a free vector $(A,B)\in \overline L$ sends a point $O\in L$ to the point $Sh_{B-A}(O)=O+B-A\in L$.
The vector space $\overline L$ is naturally embedded in the vector space $\overline{\bold L}$ and the actions of the additive groups of these spaces on $L$ agree with this embedding.

Let us choose any point $O$ in an affine space $L$ and  identify a point $A\in L$ with the free vector $\tau_O(A)$, such that $O+\tau_O(A)=A$ (in other words, one  identifies a point $A\in L$ with the free vector $\tau_O(A)$,  represented by the ordered pair of points $(O,A)$.

This identification provides an affine space $L$ with the structure of a vector space $L_O$ whose origin is the point $O$.  Different choices of the origin $O\in L$ provide $L$ with different structures of vector spaces $L_O$.

For any point $O\in L$, the  affine space $L$ can be considered as an affine subspace in the vector space $L_O$ (which as a set coincides with $L_O$).  The affine structures induced in $L$ as a subset of $L_O$ are the same for different choices of the point $O$:

\begin{enumerate}

\item vector spaces  of free vectors in different   spaces $L_O$ can be naturally identified with each other;

\item  identified free vectors from  different spaces  $L=L(O)$ realize  the same action on the space~$L$.
\end{enumerate}

The center of mass technique deals with affine spaces over arbitrary fields. In particular, it is applicable to   $n$-dimensional Euclidean spaces which are  real $n$ dimensional affine  spaces equipped with an extra structure.

In the next section, we briefly recall the definition of multidimensional Euclidean spaces. We will not use this definition later.

\subsection{Multidimensional Euclidean Spaces}\label{subsec4.4}

One can assign the structure of an  affine space $L$
to any vector  space $\bold L$: as a set, the affine space $L$ coincides with the vector space $\bold L$, the  space  $\overline L$ of free vectors in the affine space  $ L$ and the action of  the additive group of free vectors on the affine  space $ L$ have been defined above.

The real $n$-dimensional affine space $R^n$ is, by definition, the affine space  associated with the real $n$-dimensional vector space $\Bbb R^n$.

An $n$-dimensional Euclidean space is a real $n$-dimensional affine  space $L$  equipped with  a positive  definite quadratic form   on the vector space of free vectors on $L$.

Using a positive definite  quadratic form on $\overline L$, one can define  distances between points, angles between lines in $L$,  and develop geometry based on this notions  in $L$.

For $n=1,2,3$  this geometry can be identified  with the classical Euclidean geometries of corresponding dimensions.

 Euclidean geometry in the spaces of dimension $n>3$  is  a  natural and rich  generalization of classical Euclidean geometry.

\subsection{Affine Maps}\label{subsec4.5}

Let $L_1$ and $L_2$ be affine spaces over a field $\bold k$. A map $F:L_1\to L_2$ sends an ordered pair  of points  $(A,B)$ in $L_1$ to the ordered pair of points $(F(A),F(B))$  in $L_2$.

\begin{definition} A map $F:L_1\to L_2$ is an {\sl affine map}  if:
\begin{itemize}

\item   the induced map  $\overline F:\overline L_1\to \overline L_2$ is well defined, i.e., if the pairs $(A,B)$ and $(C,D)$ define the same  free vector in $\overline L_1$, then the pairs $(F(A),F(B))$  and $F((C),F(D))$  define the same  free vector in $\overline L_2$;

\item  the induced map $\overline F:\overline L_1\to \overline L_2$ is a linear map.
\end{itemize}
\end{definition}

Let us choose some points $O_1\in L_1$ and $O_2\in L_2$ and  consider affine spaces $L_1$ and $L_2$ as the vector spaces $(L_1)_{O_1}=L_1$ and $(L_2)_{O_2}=L_2$.

\begin{definition} A map  $F:(L_1)_{O_1}\to (L_2)_{O_2}$   is a {\sl shifted linear map} if it can be represented in the form $$F(x)=A x+b,$$
where   $A:(L_1)_{O_1}\to (L_2)_{O_2}$ is an arbitrary linear map (which is called the {\sl linear part of $F$}),  $b \in (L_2)_{O_2}$ is  an arbitrary  vector, and  $x$ is any point in the vector space $(L_1)_{O_1}$.
\end{definition}

\begin{theorem} The map $F:L_1\to L_2$ is an affine map if and only if, for any choice of points $O_1\in L_1, O_2\in L_2$, it gives rise to a shifted linear map of  $(L_1)_{O_1}$ to $(L_2)_{O_2}$.

Moreover, the linear part $A$  of the map $F$ can be naturally identified with the linear map $\overline F:\overline L_1\to \overline L_2$, which is independent of choice of points $O_1$ and $O_2$.
\end{theorem}

\begin{proof}  Assume that $F:L_1\to L_2$ is an affine map.  By definition, $F$ induces a linear map  $\overline F: \overline L_1 \to \overline L_2$.  Let us identify the space $\overline L_1$ with the space $( L_1)_{O_1}$ and the space $\overline L_2$ with the space $( L_2)_{F(O_1)}$. Under this identification, the map $\overline F$ can  be considered as a linear map $A: ( L_1)_{O_1}\to ( L_2)_{F(O_1)}$. Let $F(O_1)$ be a point $b\in (L_2)_{O_2}$. Then the map
$F:L_1\to L_2$ can be identified with the map $F(x)=A(x)+b$ from  $(L_1)_{O_1}$ to $(L_2)_{O_2}.$

Conversely, any map $ F(x)=A(x)+b$ of $(L_1)_{O_1}$ to $(L_2)_{O_2}$, where $A$ is a linear map and $b$ is an arbitrary  vector in $(L_2)_{O_2}$, obviously, gives rise to an affine map of $L_1$ to $L_2$.
\end{proof}

\begin{corollary} The space of polynomials  of degree $\leq 1$ on $L$   coincides with the space of  affine maps of an affine space $L$ to the one dimensional affine space $\bold k^1$.
\end{corollary}

\begin{proof} In the vector space $L_O=L$ with the origin at a point  $O\in L$, a function $F:L\to \bold k$ is a polynomial  of degree $\leq 1$ if it can be represented   as $P(x)=L(x)+b$, where $L(x)$ is a linear function, and $b\in \bold k$ is a constant. Thus, $F$ is a polynomial of degree $\leq 1$ if and only if it is a shifted linear map of $L_O$ to $\bold k$.
\end{proof}

One can automatically verify the
following two lemmas.

\begin{lemma} Let $L_1$  be an affine space   over $\bold k$, and let $L_2$ be a vector space   over $\bold k$. Then the space of affine maps $F:L_1\to L_2$  is a vector space over $\bold k$, i.e., the sum $F_1+F_2$ of affine  maps $F_1,F_2$ and the product $\lambda F$ of affine map $F$ on $\lambda\in \bold k$ are affine maps.
\end{lemma}

\begin{lemma} Let $L_1,L_2,L_3$ be affine spaces over $\bold k$ and let $F_1:L_1\to L_2$ and $F_2:L_2\to L_3$ be affine maps, then:

\begin{enumerate}
\item  the composition $F_3=F_2 \circ F_1:L_1\to L_2$ is an affine map;

\item the linear part $\overline F_3$ of the composition is  the composition $\overline F_2\circ \overline F_1$ of the linear parts $\overline F_2$ and $\overline F_1$ of the maps $F_2$ and $F_1$.
    \end{enumerate}
\end{lemma}

It is easy to verify the following theorem.

\begin{theorem} An affine map $F:L_1\to L_2$ is invertible if and only if its linear part $\overline F;\overline L_1\to \overline L_2$ is invertible.

All invertible affine maps $F:L\to L$ of an affine space to itself form a  group under composition.
\end{theorem}

\begin{definition} The {\sl group of affine transformations} of an affine space $L$ is the group of all invertible affine maps $F:L\to L$.
\end{definition}

\begin{remark}
According to Felix Klein's  Erlangen program,
in affine geometry of an affine space $L$, two subsets $X$, $Y$ are considered as {\sl equal sets} if there is an affine transformation of $L$ which sends the set $X$ to the set $Y$.

Affine geometry studies  properties of sets, which are preserved under affine transformations (i.e., which are the same for equal sets).
\end{remark}

\section{Weighted Sets and their Moment Maps}\label{sec5}

In this section, we discuss weighted sets in an affine space $L$, their moments about a pivot point and their moment maps.

\subsection{The Vector Space of Weighted Sets}\label{subsec5.1}

In this section, we define the  vector space of finite   sets of points in $L$ equipped with weights from the field $\bold k$. The  definition does not rely on the affine structure of $L$ and can be applied to any set~$L$.

Let $L$ be an affine space over an arbitrary
field $\bold k$.

\begin{definition} A weighted set $\widetilde A= \{(A_i,\lambda_i)\}$  in  $L$  is  a finite set $A=\{A_i\}$ of points in $L$ equipped  with elements $\{\lambda_i\}$ of the field $\bold k$. The element  $\lambda_i\in \bold k$ is {\sl the mass assigned to the point $A_i$}.
The {\sl total mass} of a weighted  set $\widetilde A$ is the element $\lambda\in \bold k$ equal to the sum $\sum \lambda_i$ of  masses $\lambda_i$ of all points $A_i$ in the set.
\end{definition}

The set $D(L)$ off all weighted sets in $L$ has a natural structure of a vector space over the field~$\bold k$.
Indeed, one can interpret a weighted set $\widetilde A$  as a function $f_{\widetilde A} :L\to \bold k$ which vanishes  on  $L\setminus A$ and whose value at each point $A_i\in A$ is equal to $\lambda_i$.

\begin{definition} The {\sl space $\widetilde D(L)$} is the vector space over $\bold k$  of all functions on $L$, taking values in the field $\bold k$ and equal to zero everywhere but on a finite set of points. A weighted set $\widetilde A\in D(L)$ can be identified with the function $f_{\widetilde A}\in \widetilde D(L)$. This identification provides the set $D(L)$ with the structure of a vector space over $\bold k$.
\end{definition}

 \begin{example} A {\sl  partition of a weighted set} is a representation of the set as a union of its disjoint subsets equipped with the masses induced from the weighted set.

If weighted subsets $\widetilde B$ and $\widetilde C$ of a weighted set $\widetilde A$ realize its partition, then the weighted set $\widetilde A$ is the sum of the weighted sets $\widetilde B$ and $\widetilde C$.
\end{example}

A function $\lambda_T:D(L)\to \bold k$, which assigns to a weighted set $\widetilde A$ its total mass, is a {\sl linear function} on the space $D(L)$.

\begin{definition} A  weighted set $\widetilde A$ in $L$  with {\sl total} mass $\lambda_T\in \bold k$ is called:
\begin{enumerate}
    \item  a {\sl  weighty set} if its total mass $\lambda_T$ is not equal to zero;

   \item  a {\sl  weightless set} if its total mass $\lambda_T$ is  equal to zero.
\end{enumerate}
\end{definition}

\begin{definition} Let us denote by $D_0(L)$ the subset of the set $D(L)$ that consists of all weightless sets $\widetilde A$ in $L$.
\end{definition}

Since the total mass of a  weighted set is a linear function on the vector space $D(L)$, {\sl the set of weightless sets is a vector subspace of codimension one  in $D(L)$}.

\subsection{The Moment of a Weighted Set}\label{subsec5.2}

One can  repeat almost verbatim  definitions of the moment about a pivot point and of the moment maps  of weighty  sets in an affine space over an arbitrary field $\bold k$ (see below).

\begin{definition} The {\sl moment $P(O)=\Delta_{\widetilde A}(O)$   of a weighted set $\widetilde A= \{(A_i,\lambda_i)\}$ about a pivot point $O$}  is a linear combination of the free vectors   $(O,A_i)$ with  coefficients $\lambda_i$:

\begin{equation}\label{equation: formula for moment} P(O)=\sum \lambda_i (O,A_i).
\end{equation}

\end{definition}

The following Theorem is obvious:

\begin{theorem}[Linearity of moment]
\label{additivity}

For each pivot point $O$, the moment $P(O)=\Delta_{\widetilde A} (O)$ of a weighted set $\widetilde A$ about the pivot point $O$ gives a linear map of the   space $D(L)$  to the space $\overline L$ of free vectors in $L$.
\end{theorem}

\begin{proof} To prove the theorem one has to check that:
\begin{enumerate}

\item  if  $\widetilde A=\widetilde  B +\widetilde C$,
   then, for every point $O\in L$, the
   identity $\Delta _{\widetilde A}(O)=\Delta_{\widetilde B}(O)+\Delta_{\widetilde c}(O)$ holds;

\item if $\widetilde A=\tau \widetilde B$ for $\tau\in \bold K$, then, for every point $O\in L$, one has $\Delta_{\widetilde A}(O)=\tau \Delta _{\widetilde B}(O).$
 \end{enumerate}

One can  easily verify each of these statements.
\end{proof}

\begin{corollary} For any partition  of a weighted  set into a union of weighted  subsets, the moment of the set about any pivot point is equal to the sum of the moments of the subsets about the same pivot point.
\end{corollary}

\begin{definition} Two  weighted sets $\widetilde A$ and $\widetilde B$ are {\sl equivalent}, i.e., $\widetilde A\sim \widetilde B$ if their moments about any pivot point are equal.
A weighty set in $L$ is a {\sl null set} if its moment about any pivot point is equal to zero.
Denote the set  of all null sets in $L$   by $DM(L)\subset D(L)$.
\end{definition}

By definition, two sets $\widetilde A$ and $\widetilde B$ are equivalent if $\widetilde A-\widetilde B$ is a null set.

Theorem \ref{additivity} implies the  following corollary.

\begin{corollary}\label{equivalence} The set $DM(L)$ of all null sets in $L$ is a vector subspace of the vector space $D(L)$.  The  equivalence of weighted sets respects the linear operations on weighted sets, i.e., the following relations hold:
\begin{enumerate}

 \item if  $\widetilde A\sim \widetilde B$, then
     $\mu \widetilde A \sim \mu\widetilde B$ for any $\mu \in \bold k$;

\item  if $\widetilde A_1 \sim \widetilde A_2$ are $\widetilde B_1 \sim \widetilde B_2$, then   $\widetilde A_1+\widetilde B_1 \sim \widetilde A_2+\widetilde B_2$.
    \end{enumerate}

\end{corollary}

\begin{definition} The moment map $\Delta=\Delta_{\widetilde A}$   of a weighted set $\widetilde A$ is the map $\Delta:L\to \overline L$ which sends a point $O \in L $ to the moment $\Delta_{\widetilde A} (O)$ of the set $\widetilde A$ about the pivot point $O$.
\end{definition}

The following theorem repeats Theorem \ref{change of pivot} almost verbatim and can be checked in the same~way.

\begin{theorem}[Change of a pivot point]\label{change of pivot}
Let $P=\Delta_{\widetilde A}$ be the  moment map  corresponding to a weighted set $\widetilde A$,
whose total mass is equal to $\lambda=\lambda_T(\widetilde A)$.
Then, for any two points $O_1, O_2\in L$, the map $P$  satisfies the following  {\sl relation}:
\begin{equation}\label{characteristic}
P(O_1)-P(O_2)=-\lambda (O_2O_1).
\end{equation}
\end{theorem}

\section{The Space of Moment-like Maps and the  Space of Weighty Points and Mass Dipoles}\label{sec6}

In this section, we study the space of moment-like maps and define the space  of weighty points and mass dipoles in an affine space $L$ over an arbitrary field $\bold k$.

\subsection{Moment-like Maps}\label{subsec6.1}

In this section, we will study a special class of affine maps which plays a key role for  the center of mass technique.

\begin{definition} An affine map $P$ of an affine space $L$ over a field $\bold k$  to the vector space $\overline L$  of free vectors in $L$  is a {\sl moment-like map}  if its linear part $\overline P:\overline L\to \overline L$ is equal to $-\lambda I$, where  $I:\overline L\to \overline L$ is the identity map and  $\lambda\in \bold k$ is a parameter  called {\sl the total mass} of  $P$ and  denoted by $\lambda_T(P)$.
Denote by $M(L)$ the collection of all moment-like maps of an affine space $L$ to the vector space $\overline L$ of free vectors on $L$.
\end{definition}

The definition implies that  a map $P:L\to \overline L$ is a moment-like map with total mass $\lambda_T (P)$ if and only if,
for any two points $O_1,O_2\in L$, the following  {\sl characteristic identity} holds:
\begin{equation}\label{moment is affine}
P(O_2)-P(O_1)=-\lambda_T(P) (O_2-O_1).
\end{equation}

\begin{lemma} The set $M(L)$ of all moment-like maps of an affine space $L$ over a field $\bold k$ is a vector space over $\bold k$, i.e., the sum of two moment-like maps of $L$ is a moment-like map of $L$; the  product of a moment-like map of $L$ on an element $\lambda \in \bold k$ is a moment-like map of $L$.

The function $\lambda_T: M(L)\to \bold k$, which assigns to a map $P\in M(L)$  its total mass $\lambda_T(P)$, is a linear function on   $M(L)$.
\end{lemma}

\begin{proof} For each fixed pair of points $O_1, O_2\in L$ the characteristic identity  can be considered as a linear homogeneous equation on a map $P:L\to \overline L$. Solutions of any system of linear homogeneous equations form a vector space.

Linearity of the function $\lambda_T:M(L)\to \bold k$ is also a straightforward  consequence of definitions.
\end{proof}

The characteristic identity can  be rewritten in the following form:
\begin{equation}\label{moment is affine 2}
P(O_2)= P(O_1) -\lambda (O_2 - O_1).
\end{equation}

Formula (\ref{moment is affine 2})  immediately implies the following theorem.

\begin{theorem} The total mass  of a map $P\in M(L)$  is equal to zero if and only if  $P$ is a constant map $P\equiv v$, where  $v\in \overline L$.
The space $M_0(L)\subset M(L)$ of all maps $P\in M(L)$, whose  total mass equals to zero, is a vector subspace in  $M(L)$  of codimension one.
The map $\tau: M_0(L)\to \overline L$, that sends  $P\equiv v$ to the free vector $v$, is a natural isomorphism between the space $M_0(L)$ and the space~$\overline L$.
\end{theorem}

\begin{proof}  According  (\ref{moment is affine 2}), a map $P\in M(L)$  is a constant map if and only if its total mass is equal to zero.
The total mass  $\lambda_T:M(L)\to \bold k$ is a linear function on  $M(L)$.  Thus, the equation $\lambda_T(P)=0$ defines a vector subspace  $M_0(L)$ of $M(L)$ whose  codimension  is one.

The second statement of the theorem is obvious.
\end{proof}

\begin{lemma}\label{mass dipole of constant map}  A constant map $P\equiv v$ is the moment map $\Delta_{\widetilde A}$ of a mass dipole  $\widetilde A=\{-O,B\}$ such that the ordered pair of points $(O,B)$ defined up to a shift, represents the  free vector $v$.
\end{lemma}

\begin{proof} The lemma can be proven   in almost  the same way as Lemma \ref{moment map of mass dipole}.
\end{proof}

\begin{lemma} There is a unique  map $P\in  M(L)$, whose total mass is a given element $\lambda\in \bold k$ and whose value $P(O_1)$  at a chosen  point $O_1\in L$ is a given free vector $v\in \overline L$.
\end{lemma}

\begin{proof} According to (\ref{moment is affine 2}), the value  $P(O_2)$ of the map $P$ at a point $O_2\in L$  is equal to $v -\lambda (O_2-O_1)$.

Conversely, the map $P$ defined by the  formula $P(O_2)= v -\lambda (O_2-O_1)$ satisfies the assumption of the lemma.
\end{proof}

\begin{corollary} The vector space $M(L)$ is isomorphic to the direct sum $\overline L \oplus \bold k^1$ (an isomorphism is not canonical).
\end{corollary}

\begin{proof} Fix a point $O\in L$. The map $\tau_O :M(L)\to \overline L \oplus \bold k^1$ that sends a map $P\in M(L)$ to the pair $(P(O),\lambda_T (P))$  for each point $O\in L$, provides  an isomorphism between $M(L)$ and $\overline L \oplus \bold k^1$. This isomorphism depends on the point $O\in L$ and is not canonical.
\end{proof}

\begin{corollary}  The dimension of the space $M(L)$ is equal to the dimension of the space $L$ plus one, i.e., the dimension $\dim_{\bold k} M(L)$ of the space $M(L)$  is equal to $\dim_{\bold k}\overline L + 1 =
\dim_{\bold k} L+ 1$.
 \end{corollary}

\begin{definition} Let $P\in M(L)$ be a map whose total mass  is not equal to zero. The {\sl center of mass  of the map} $P$ is a point $O\in L$  at which the map $P$ vanishes, i.e.,  $P(O)=0$.
\end{definition}

\begin{definition} The {\sl  normalized  map}  of a map $P\in M(L)$ whose total mass $\lambda$ is not equal to zero, is the map $\lambda^{-1} P$.
\end{definition}

\begin{lemma}  A map  $P\in M(L)$ with nonzero total mass $\lambda$ and the normalized map $\lambda^{-1}P$
vanish at the same point.
The normalized  map $\lambda^{-1} P$  has the  total mass 1.
\end{lemma}

\begin{proof} Two proportional vector-valued functions  $P$ and $\lambda^{-1}P$ vanish at the same points. The total mass  is a linear function on the space $M(L)$. The total mass  of the normalized map  is equal to $\lambda^{-1}\cdot \lambda=1\in \bold k$.
\end{proof}

\begin{theorem}\label{total mass} Assume that the total mass  of a  map $P\in M(L)$ is equal to $1\in \bold k$. Then,  for any point  $Q\in L$, the free vector $P(Q) \in \overline L$ sends the point $Q$ to the point $O=Q +P(Q)$ which is independent of the point $Q$.

Moreover,  the  point $O$ is the unique point at which the function $P$  vanishes.
\end{theorem}

\begin{proof}  Indeed, for any two points $Q_1, Q_2\in L$  the characteristic identity $P(Q_1)-P(Q_2)=Q_2-Q_1$ can be rewritten in a following form: $Q_1+P(Q_1) =Q_2+P(Q_2)$. Thus, the points $Q_1$ and $Q_2$ are mapped by the free vectors $P(Q_1)$ and $P(Q_2)$ to the same point $O=Q_1+P(Q_1) =Q_2+P(Q_2)$.

The point $O$ itself  has to be sent to the point $O$, so
 that$O=O+P(O)$ which implies  that $P(O)=0$.  For any other point $Q\neq O$, we have
$Q+P(Q)=O\neq Q$. Thus, $P(Q)\neq 0$.
The theorem is proven.
\end{proof}

\begin{theorem}\label{ceter of mass is unique} Assume that the total mass $\lambda$ of a  map $P\in M(L)$  is not equal to zero. Then $P$ has a unique center of mass $O$.  For any point  $Q\in L$ the following relation holds: the free vector $\lambda^{-1}P(Q) \in \overline L$ sends the point $Q$ to the center of mass $O$, i.e.,
\begin{equation}\label{tilde P}
O=Q + \lambda ^{-1}  P(Q).
\end{equation}
Moreover, if the affine space $L$ has a structure of   a liner space $\bold L$, then  the vector $O\in\bold L$ satisfies the following relation:
\begin{equation}\label{tilde O}
O= \lambda^{-1} P(\bold O),
\end{equation}
where $ \bold O \in \bold L$ is the zero in the space $\bold L$.
\end{theorem}

 \begin{proof} The theorem follows from Theorem \ref{total mass} applied to the normalized  map $\lambda^{-1} P$.
Formula (\ref{tilde O}) can be obtained  by plugging  $Q=\bold O$ into formula (\ref{tilde P}).
 \end{proof}

\begin{theorem}\label{center of mass and total mass} A map $P\in M(L)$ with nonzero total mass is uniquely determined by its center of mass $O$ and by its total mass $\lambda$.

Moreover, for any point $Q\in L$, the following identity holds: $P(Q)=-\lambda(O-Q)$. Thus,  $P$ is the moment map of the set $\widetilde A$ consisting of one point $O$ equipped with mass $\lambda$.

Conversely, the moment map $P=\Delta_{\widetilde A}$ of the set $\widetilde A=\{(O,\lambda)\}$ has total mass  $\lambda$ and  center of mass  $O$.
\end{theorem}

\begin{proof} Two maps from the space $M(L)$ having the same center  of mass  $O$ and the same total mass coincide, since their values at the  point $O$ are equal, and their total masses are the same.

The  point $O$, obviously, is the center of mass of the map $P\in M(L)$ given by the relation $P(Q)=-\lambda(O-Q)$, and  total mass  of $P$ is $\lambda$.
\end{proof}

\begin{lemma} If two maps $P_1, P_2 \in M(L)$   are equal at two different points $O_1, O_2\in L$, then they are equal identically.
\end{lemma}

\begin{proof} Indeed, any non-zero  map $P\in M(L)$ vanishes no more than at one point. The maps $P_1 -P_2$ vanishes at two different  points $O_1$ and $O_2$. Thus, $P_1-P_2$ is identically  equal to zero.
\end{proof}

\subsection{The Space $M(L)$ and the Space $D(L)$ of Weighted Sets in $L$}\label{subsec6.2}

In this section, we show that the space $M(L)$ of moment-like maps is isomorphic to space of equivalence classes of weighted sets, i.e., is isomorphic to the space $D(L)/DM(L)$ (see Corollary~\ref{equivalence}).

\begin{definition} The {\sl  moment correspondence}  is the map $\Delta : D(L)\to M(L)$ which sends each weighted set $\widetilde A\in D(L)$ to its moment map $\Delta_{\widetilde A}\in M(L)$.
\end{definition}

The kernel of the moment correspondence is the space $DM(L)\subset D(L)$ of null sets in $L$. It is easy to see that a weighted set $\widetilde A$ is a null set if and only if  its total mass and its mass dipole are equal to zero.

\begin{theorem} The moment correspondence is a surjective  map, i.e., any moment-like map is a moment map of some weighted set $\widetilde A$.
The space $M(L)$ is isomorphic to a factor-space $D(L)/DM(L)$ of the space $D(L)$ by its  subspace $DM(L)$.
 \end{theorem}

\begin{proof} Assume that the total mass $\lambda$ of a moment-like map $P\in M(L)$ is not equal to zero, and its center of mass is a point $O$.  Then, by Theorem \ref{center of mass and total mass},  $P=\Delta_{\widetilde A}$, where $\widetilde A$ consists of one point $O$ equipped with mass $\lambda$.

 Assume that the total mass $\lambda$ of a moment-like map $P\in M(L)$ is equal to zero, that is  $P$ is a constant map $P\equiv v$.  Then, by Lemma \ref{mass dipole of constant map}, $P=\Delta_{\widetilde A}$, where $\widetilde A$ is a mass dipole $\{-O,B\}$, where $(O,B)$ is an ordered pair of points that represents the free vector $v$.

 Thus, the correspondence map $\Delta$ is  onto. Since its kernel is the space $DM(L)$ of null sets in $L$,
the space $M(L)$ is isomorphic to the factor-space $D(L)/DM(L)$.
\end{proof}

\section{The Space $\widehat M(L)$ of weighty Points and Mass Dipoles in $L$}\label{sec7}

Let $L$ be an affine space over an arbitrary field $\bold k$. In this section, we define the vector  space $\widehat M(L)$ over $\bold k$  of weighty points and mass dipoles in $L$. The additive group of that vector space has many applications in geometry.

The space  $\widehat M(L)$ can be interpreted as the factor-space $D(L)/DM(L)$ of the space $D(L)$ of weighted points in $L$ by its subspace $DM(L)$ of all null sets in $L$. Our presentation is based on the isomorphism between  the spaces $D(L)/DM(L)$ and $M(L)$ discussed above and on the properties of  the space $M(L)$.

\begin{definition} A {\sl weighty point $\{ (O,\lambda)\}$  in $L$} is the point $O\in L$ equipped with a nonzero mass~$\lambda\in \bold k$.
\end{definition}

\begin{definition} A {\sl mass dipole} $\{-A,B\}$ in $L$ is an ordered pair of points $(A,B )$ in $L$  equipped with masses $-1,  1$ respectively. Two mass dipoles $\{-A,B\}$ and $\{-C, D\}$ are called equal if the ordered pairs of points $(A,B)$ and $(C,D)$ are equal up to a shift.
\end{definition}

By definition, a mass dipole $\{-A,B\}$ can be identified with a free vector, represented by the ordered pair of points $(A,B)$ in $L$. The class of the empty set  can in $D(L)/DM(L)$ be considered as a mass dipole $\{-A,A\}$ or as the zero free vector.

\begin{definition} The set $\widehat M(L)$ defined as the set  of all weighty points in $L$ and all mass dipoles in $L$ defined up to a shift.
\end{definition}

There is a natural one-to-one correspondence
$\tau :\widehat M(L)\to M(L)$ between the sets $\widehat M(L)$ and $M(L)$
(see Definition \ref{correspondence} below).

The correspondence $\tau: \widehat M(L)\to M(L)$ allows to identify the sets $\widehat M(L)$ with the space $M(L)$ and to define a structure of a vector space on the set $\widehat M(L)$.

\begin{definition} The  structure of a vector space over $\bold k$ on $\widehat M(L)$ is the structure induced by  the identification  $\tau:\widehat  M(L)\to M(L)$ from the   $\bold k$-vector space structure on  $M(L)$.
\end{definition}

The above definition implies that the   vector space operations on the set $\widehat M(L)$ are defined as follows. To perform a vector space operation on points from the set $\widehat M(L)$ one has to:

\begin{itemize}

\item  identify the  points in $\widehat M(L)$  with the maps from the space $M(L)$;

\item   perform the desired vector space operation on the maps from $M(L)$, and

\item    identify  the map, obtained as the  result of this operation,  with a point of $\widehat M(L)$.
\end{itemize}

We will  specify vector space operation on the space $\widehat M(L)$ in Section \ref{subsec7.1}. Now we will define the one-to-one correspondence $\tau:\widehat M(L)\to M(L)$.

The space $M(L) $ contains maps of two  types: maps of the  {\sl first type} whose total mass is a nonzero element of the field $\bold k$,  and  maps of the {\sl second type} whose total mass is equal to zero.

Each map of the  first type is uniquely determined by its center of mass $O$   and by its total mass $\lambda\neq 0$.

Each map of the second type is a constant map $P\equiv v$, where $v\in \overline L$ is a free vector represented by an ordered pair of points $(A,B)$ defined up to a shift.

We now define the map $\tau$.

\begin{definition}\label{correspondence} The map $\tau$ sends a weighty point $\{(O,\lambda)\}$ to the map $P\in M(L)$ of {\sl first type}  whose  total mass is $\lambda$ and whose center of mass is the point $O$.

The map $\tau$ sends a mass dipole $\{-A,B\}$ to the map $P\equiv v$, $P \in M(L)$ of the {\sl second  type},  where  $v$ is a free vector, corresponding to an ordered pair of points  $(A,B)$.
 \end{definition}

Now, that vector space operations on the set $\widehat M(L)$ are defined, one can describe each type of vector space operation
separately, see lemmas in Section \ref{subsec7.1}  below. Each of these lemmas can be easily verified. The content of the lemmas can be considered as an axiomatic definition of the space $\widehat M(L)$. The additive  group of the space $\widehat M(L)$ is important for geometrical applications.

Since the set $\widehat M(L)$ contains points of  very different nature, definition of addition  is a bit tricky: addition of  different   types of  pairs of points in $\widehat M(L)$ is performed  by different rules. Because of that, the associativity of addition is not obvious.

Geometrical applications of the additive group $\widehat M(L)$  come from the following statement:

\noindent{\it The sum of several elements of the group $\widehat M(L)$ is well defined: it does not depend on  order  in which  the addition  of pair of elements is performed}.

\subsection{The Vector Space Operations on  $\widehat M(L)$}\label{subsec7.1}

Let us describe vector space operation on the space $\widehat M(L)$.

\begin{lemma} The product of a weighty point $\{(O,\lambda)\}$ by $\mu\in \bold k$,  is the weighty point $\{(O,\mu \lambda)\}$ if $\mu\neq 0$, and is a zero mass dipole $\{-A,A\}$ if $\mu =0$.
The product of a mass dipole  $\{-A,B\}$, corresponding to a free vector $(A,B)$, by $\mu\in \bold k$, is the mass dipole, corresponding to the free vector $\mu(B-A)$.
\end{lemma}

In geometrical applications, the additive  group
of the space $\widehat M(L)$ plays  a key role. Let us give more detailed description of addition  in $\widehat M(L)$.

\begin{lemma}\label{two weighty points} The sum of two weighty points $\{(A,\mu)\}$ and $\{(B,\rho )\}$ with $\mu+\rho=\lambda \ne 0 $ is the weighty point $\{(O,\lambda )\}$, where $O$ is a point such  that the free vectors $(A,O)$ and $(O,B)$ are proportional and
their ratio $(A,O):(O,B)$ satisfies the relation
\begin{equation}\label{formula for axiom 1}
(A,O):(O,B)=\frac{\rho}{\mu}.
\end{equation}
\end{lemma}

\begin{proof} By definition, the center of mass   of the weighty set defined in the statement of the  lemma,  is a point $O$ such that $\mu (O,A)+\rho (O,B)=0$, or $\mu (A,O)=\rho ( O,B)$ which is equivalent to (\ref{formula for axiom 1}).
\end{proof}

If $L$ is a real affine space, then Lemma \ref{two weighty points} implies that the oriented segments $AO$ and $OB$ are proportional and
$$AO:OB=\frac{\rho}{\mu}.$$

Lemma \ref{two weighty points}  implies that the point $O$ as a vector in a vector space  $\bold L$ containing the affine space $L$,
satisfies  the following relation:
 $$O=\frac{\mu}{\lambda}A + \frac{\rho}{\lambda}B.$$

Let us justify the classical center of mass technique.

\begin{proof}[Proof of Theorem \ref{theorem 1.33}]
We have to prove that there is a way of assigning the center of mass to every weighty set that satisfies Axioms 1, 2. The uniqueness of such center of mass is proved above (Theorem \ref{induction}). Above, we defined the center of mass of a weighty set  $\widetilde A$ as the only point where the moment map $\Delta_{\widetilde A}$ vanishes. Lemma \ref{two weighty points} implies that the center of mass defined in such a way satisfies Axiom 1. It satisfies Axiom 2, since the moment map  $\Delta_{\widetilde A}$, by Theorem \ref{additivity},   depends linearly on the weighted set $\widetilde A$.
\end{proof}

The next three lemmas describing other types of addition in the group $\widehat M(L)$, are  also straightforward. We will nor present their proofs.

\begin{lemma}\label{point and mass dipole} The sum of a weighty point $\{(A,\lambda)\}$ and a mass dipole $\{-B,C\}$ is the weighted point $\{ (O, \lambda)\}$,  where the point $O$  is the sum $A+\lambda^{-1} (B,C)$  of the point $A$ and the free vector $\lambda^{-1} (B,C)$.
\end{lemma}

\begin{figure}[htbp]
\begin{center}
 \includegraphics[scale=0.35]{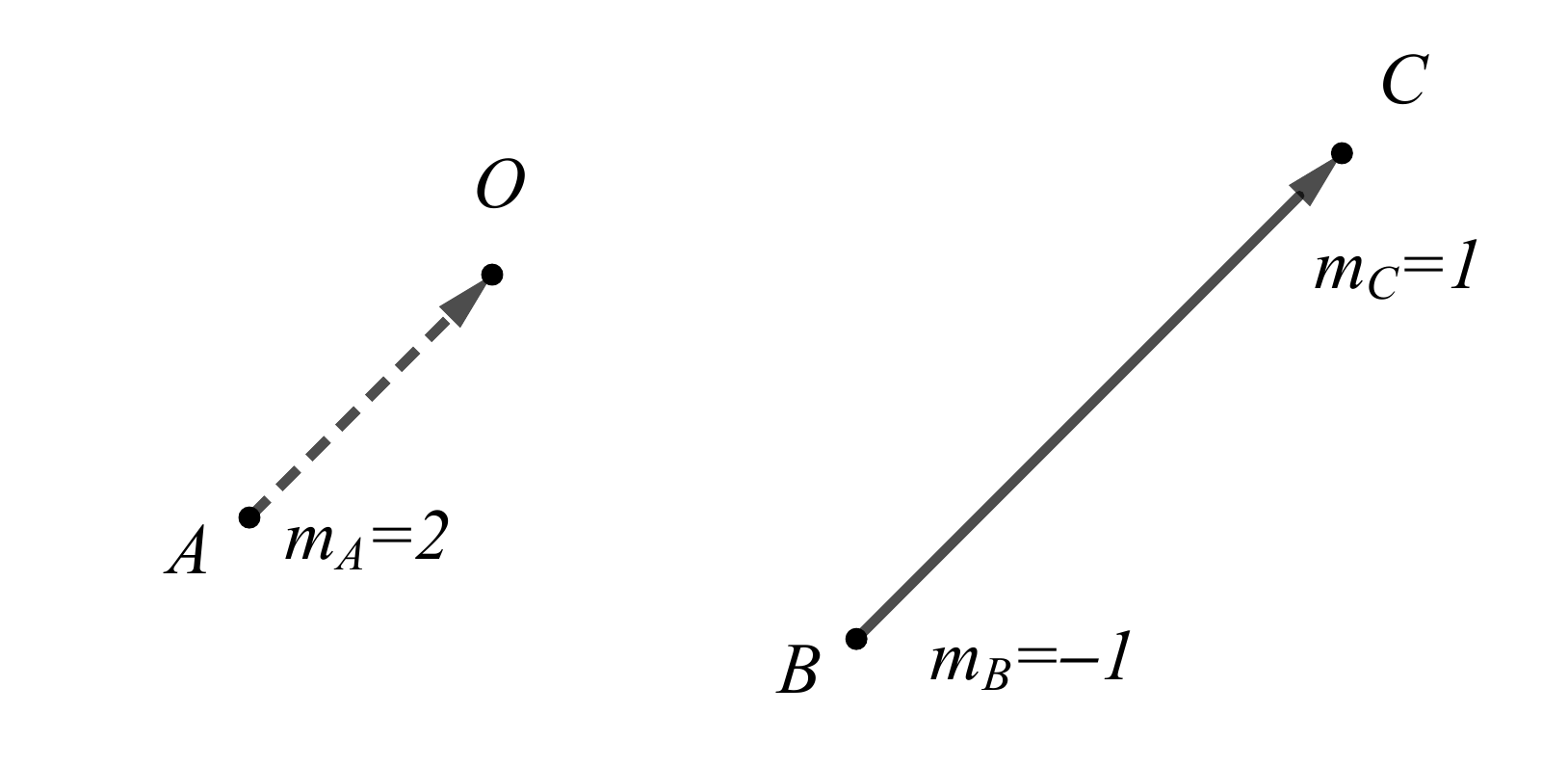}
\end{center}
\caption{The center of mass of a weighted set $\{(A,2),(B,-1),(C,1)\}$
 is the point $O$, where $2AO=BC$ }
\label{fig:point and dipole}
\end{figure}

If $L$ is a real affine space, then the lemma implies that the oriented segments  $AO$ and $\lambda^{-1} (B,C)$ are parallel, equal and point to the same direction (see Figure \ref{fig:point and dipole}).

The lemma implies that the point $O$ as  vectors in a vector space  $\bold L$ containing an affine space $L$, satisfy  the following relation:
 $$O=A+\lambda ^{-1}(C-B).$$

\begin{lemma} The sum of two mass dipoles $\{-A,B\}$ and $\{-C,D\}$ is the mass dipole $\{-E,F\}$  that corresponds to the free vector $(A,B)+(C,D)$.
\end{lemma}

If $L$ is a real affine space, then the lemma implies that the oriented segments $AB, CD, CF$ (which we consider us free vectors defined up to shifts) satisfy the following relation:
$$EF=AB+CD.$$

The lemma implies that the points $E,F$ as  vectors in a vector space  $\bold L$ containing an affine space $L$, satisfy  the following relation:
 $$F-E=(B-A)+(C-D).$$

\begin{lemma} The sum of two weighted points $\{(A,\mu)\}$ and $\{(B,-\mu )\}$  is the mass dipole $\{-C,D\}$  that corresponds to the free vector $\mu (A,B)$.
\end{lemma}

If $L$ is a real affine space, then the lemma implies that the oriented segments $CD, \mu AB$,  defined up to a shift, satisfy the following relation:
$$CD=\mu AB.$$

The lemma implies that the points $C,D$ as  vectors in a vector space  $\bold L$ containing an affine space $L$, satisfy  the following relation:
$$D-C=\mu (B-A).$$

\section{The Orthocenter and the Euler Line}\label{sec8}

In this section, we give a simple application of the extended center of mass technique (in other words, of the additive group $\widehat M(L)$). We will prove  the following classical results:

\begin{enumerate}

\item the three altitudes of a triangle have a common point; it is called the {\sl orthocenter};

\item for any triangle, its orthocenter $H$, its barycenter $M$ (i.e., the  intersection points of its medians) and its  circumcenter  $O$ (i.e., the center $O$ of  a circle, passing through all the vertices of  the triangle) belong to one line (which is called the {\sl Euler line} of the triangle). Moreover, the relation $HM:OM=2:1$ holds.

\end{enumerate}

\subsection{The Three Altitudes of a Triangle and Masses}\label{subsec8.1}

Applying the extended   center of mass technique, one can show that the three altitudes of a triangle pass through a common point (which is called the {\sl orthocenter of the triangle}). Moreover, one can see how the orthocenter is located at each of these altitudes (see Theorem \ref{theorem8.1}).

Consider a  triangle $ABC$. Let $O$ be the center of the circle passing through the vertices $A,B,C$.

 Consider a weighted set $\widetilde T$ containing the points $A,B,C,O$, in which the points $A,B,C$ have mass $1$ and the point $O$  has mass $-2$.

Let us compute the center of mass of the set $\widetilde T$  in several different ways.

Let $CC'$ be  the altitude of $ABC$ passing through the vertex $C$ and let $O_{AB}$ be the midpoint of the side $AB$.

\begin{lemma} Denote by $H_C$  the point on the line $CC'$ such that the oriented segment $C H_C$ is equal to the oriented segment $O O_{A,B}$  multiplied by two, i.e.,  $C H_C=2 O O_{AB}$. Then  the point $H_C$ is the center of mass of weighty set $\widetilde T$.
\end{lemma}

\begin{figure}[htbp]
\begin{center}
 \includegraphics[scale=0.45]{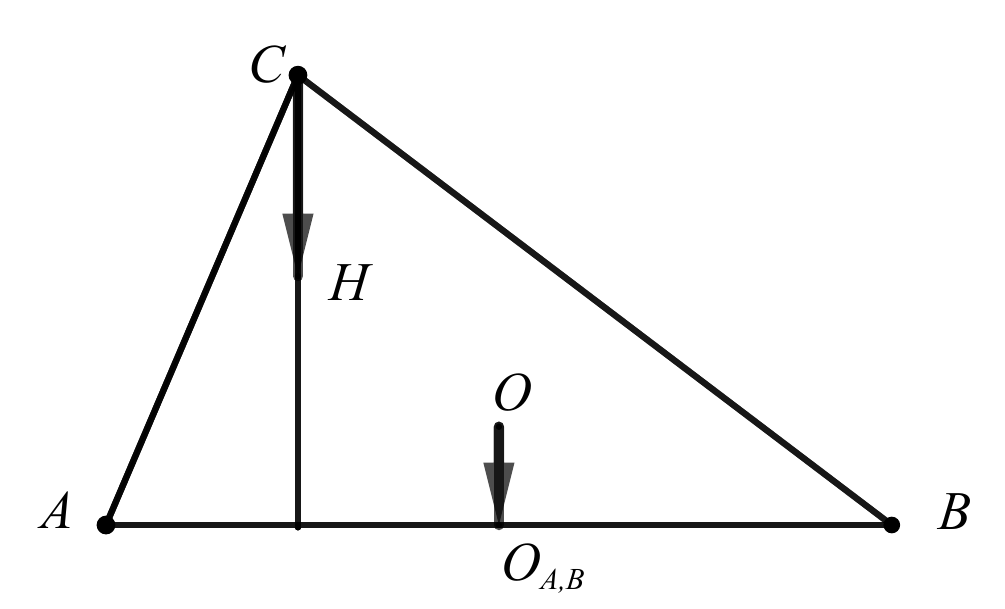}
\end{center}
\caption{If the masses at $A,B,C$ are equal to 1, and the mass at $O$ is equal to $-2$, then
the center of mass is the point $H$ such that $CH=2 OO_{A,B}$ }
\label{fig:orthocenter}
\end{figure}

\begin{proof} Partition $\widetilde T$  into two subsets:  $\{A,B, O\}$ and $\{C\}$.

The sum $1+1-2$ of the masses of the points $A,B,O$ is equal to zero. The mass dipole of these weighted points is equal to the sum of vectors $OA+OB$ which  is equal to the vector $2\cdot  OO_{AB}$. Since the point $C$ has mass~$1$, the center of mass of the set $\widetilde T$ is  at the point $C$ shifted by the vector  $2\cdot  OO_{AB}$ (see Lemma \ref{point and mass dipole}). Thus, it is equal to the point~$H_C$.
\end{proof}

In the same way, one can compute the center of mass of the set $\widetilde T$ by its subdividing it into two weighted subsets:   a weighty  vertex of the triangle and a weightless set  complimentary to the vertex. Since the center of mass is well defined, we obtain the following result.

\begin{theorem}[The  three altitudes theorem]\label{theorem8.1} In a  triangle $ABC$ the three altitudes pass through a common point $H$.
Moreover, for any vertex $V$ (where $V$ is $A$, $B$ or $C$), the vector $VH$ is equal to $2\cdot OV'$, where $V'$ is the midpoint of the side opposite to the vertex $V$.
\end{theorem}

\subsection{The Euler line of a Triangle and Masses}\label{subsec8.2}

In the previous section, we considered the weighty set $\widetilde T$ consisting of the vertices $A,B,C$ of a triangle equipped with the unit masses and  the center  $O$ of the circle passing through the vertices $A,B,C$ equipped with  mass $-2$. We  proved that the center of mass of the set $\widetilde T$ is the orthocenter $H$. In the  proof, we used three different subdivisions of the set  $\widetilde T$.

 Let us subdivide  $\widetilde T$ in a fourth  way and  apply this subdivision to another computation of  the center of mass of $\widetilde T$. It will allow us to prove a beautiful {\sl  Euler line theorem} for the triangle $ABC$.

Let $H$ be the orthocenter of the triangle $ABC$, i.e., the  point of intersection of its three altitudes.

Let $M$ be the barycenter of the triangle $ABC$, i.e., the point of intersection of its  three medians
(see Section \ref{Archimedes}).

Let $O$ be  circumcenter  of the triangle $ABC$.

\begin{theorem}[The Euler line] For any triangle $ABC$, the points $H, M, O$ are collinear. Moreover, the point $M$ divides the segment $HO$ in proportion $2:1$, i.e.,
$$ HM: MO=2:1.$$
\end{theorem}

\begin{proof} Let us subdivide the weighted set $\widetilde T$ into weighted sets  $\{A,B,C\}$ and $\{O\}$. Instead of the weighted  set $\{A,B,C\}$, one can take its center of mass $M$ equipped with mass $1+1+1=3$. Thus, the center of mass $H$ of $\widetilde T$ is equal to the center of mass of the weighted set $\{M,O\}$, where $M$ has mass $3$ and $O$ has mass $-2$.

From the properties of the center of mass, we conclude that $$3 \cdot MH=-2 \cdot HO.$$
 We also  have the following relations:  $$MH=-HM \quad \text{and} \quad HM+MO=HO.$$
These identities imply the theorem.
\end{proof}

\section{Affine Hyperplanes not Passing Through the Origin in a Vector Space}\label{sec9}

In this section we  discuss geometry of an affine hyperplane not passing
through the origin in a vector space.
We also show that any affine space can be canonically embedded into a vector space
as an affine hyperplane not passing through the origin.

\subsection{Affine Transformations  of a Hyperplane and Linear Transformations of the Ambient  Space}\label{subsec9.1}

In this section, we discuss a natural relation between affine geometry of hyperplane, not passing through the origin, and linear algebra of the ambient  vector space.

Let $L$ be an affine hyperplane in a vector space $\bold L$ not passing through the origin $\bold O$ of $\bold L$.  Let $\widetilde L$ be a vector subspace parallel to $L$. The space $\widetilde L$ can be naturally identified with the vector space $\overline L$ of free vectors in the hyperplane $L$.

\begin{theorem}\label{affine maps as linear maps} Any affine map $F:L\to L$  extends to a unique linear map $\widehat F:\bold L \to \bold L$. The affine space $L$ and the vector space $\widetilde L$  are invariant subspaces for  $\widehat F$.

Conversely, let  $G:\bold L\to \bold L$ be  a linear map such that $L$ is invariant under  $G$.
Then the  restriction $G|_L$ of $G$ to $L$ is an affine map $G|_L:L\to L$.
\end{theorem}

\begin{proof} The space $\overline L$ of free vectors in $L$ can be identified with the space $\widetilde L$. Under this identification, the linear part $\overline F: \overline L\to \overline L$ of an affine map $F:L\to L$  becomes a linear map $\widetilde F:\widetilde L\to \widetilde L$. Let us define  $\widehat F:\bold L \to \bold L$ as the linear map whose restriction to $\widetilde L$ is equal to $\widetilde F$  and whose  value  $\widehat F(A)$ at a chosen  vector $A\in  \bold L$ is equal to the vector $F(A)\in L\subset \bold L$.

It is easy to see that the linear map $\widehat F$, thus defined is an extension of $F:L\to L$. The uniqueness of an extension is obvious, since any vector in $\bold L$ can be represented as a linear combination of vectors from $L$.

Conversely, it is easy to see that the restriction of $G$ to an invariant subspace $L$ is an affine map $G_L:L\to L$.
\end{proof}

\begin{definition} Let $GL(\bold L,L)$ be the subgroup of the group $GL(\bold L)$ consisting  of invertible linear transformations of   $\bold L$ under which  the hyperplane $L$ is invariant.
\end{definition}

\begin{theorem}\label{affine transform} The restriction of a transformation $G\in GL(\bold L,V)$ to the $L$ is an affine automorphism of  $L$.
Any  affine automorphism of  $L$ is the restriction of a unique linear  transformation $G\in GL(\bold L,V)$.
\end{theorem}

\begin{proof}  The theorem follows from Theorem \ref{affine maps as linear maps}. We only have to show that if an affine map $F:L\to L$ is invertible, then its extension $\widehat F:\bold L\to \bold L$ also is invertible.

Indeed, the linear part $\overline F:\overline L\to \overline L$ is invertible. Thus, the restriction of $\widehat F$ to the invariant subspace  $\widetilde L$ is invertible. The dimension of the factor-space $\bold L/\widetilde L$ is equal to one, and the image of  $L$ spans the space $\bold L/\widetilde L$. The map $\widehat F$ induces the identity transformation the factor-space $\bold L/\widetilde L$, since it sends $L$ to $L$. Thus, $\widehat F$ is an invertible map.
\end{proof}

\begin{remark} According to Felix Klein's  Erlangen program,
geometry of an affine space $L$ is determined by the group of affine transformations of the space $L$.

Theorem \ref{affine transform} provides a  simple description of this group for  affine hyperplanes in a vector space. This description is applicable to any affine spaces, since any affine space can be canonically embedded as a hyperplane in a vector space (see Theorem \ref{affine embedding} below).
\end{remark}

\begin{theorem}\label{degree one polynomials} Any polynomial  $F$  of degree $\leq 1$  on an affine hyperplane $L\subset \bold L$ in a vector space $\bold L$  can be uniquely extended to linear function  $\widehat F$ on the space $\bold L$.

Conversely, the restriction  of any linear function on $\bold L$ to  $L$ is a polynomial of degree $\leq 1$.
\end{theorem}

\begin{proof} A polynomial $F$ of degree $\leq 1$ on an affine space $L$ is an affine map $F:L\to\bold k^1$. The linear part $\overline F:\overline L\to \bold k$ is a linear function on $\overline L$ which can be considered as a linear function on the space $\widetilde L$. Let us choose any point $A\in L\subset \bold L$. An extension $\widehat F$ can be constructed as the linear function on $L$ that coincides with $\overline F$ on the $\widetilde L$ and whose value a chosen  point $A\in L\subset \bold L$ is  $F(A)$.

The uniqueness of linear extension is obvious, since any vector  from $\bold L$ can be represented as a linear combination of elements of $L$.
\end{proof}

\begin{corollary}\label{dual to polynomials} One can naturally  identify the ambient vector space $\bold L$ with  the vector space   $P_1^*(L)$ dual  to the space $P_1(L)$ of degree $\leq 1$ polynomials on~ $L$.
\end{corollary}

\begin{proof} Theorem \ref{degree one polynomials} allows to canonically identify the space $P_1(L)$ with the space $\bold L^*$ dual to the space $\bold L$. This identification canonically identifies the  space $P_1 ^*(L)$ dual to the space $P_1(L)$ with the space $\bold L$.
\end{proof}

\subsection{Canonical Realization of an Affine Space as a Hyperplane in Vector Space}\label{subsec9.2}

Let us describe a  canonical realization of an affine space $L$ as a hyperplane, not
passing through the origin,  in a vector space.

 We will need a general definition of Kodaira's map. Let $X$ be a set, and let $W$ be a space of functions on $X$  taking values in a field $\bold k$.

 \begin{definition} {\sl Kodaira's map} $Kod :X\to W^*$ is the  map of $X$ to the dual space $W^*$ that sends  a point $A\in X$ to the linear function on $W$ whose value at  $f\in W$ is equal to  $f(A)$.
 \end{definition}

\begin{lemma}\label{embedding} Kodaira's map is an embedding if  $W$  separates points in $X$, i.e.,  for any points $A\neq B$ in $X$, there is a function $f_{A,B}\in W$ such that $f_{A,B}(A)\neq f_{A,B}(B)$.
\end{lemma}

\begin{proof} If $W$ separates points  $A,B\in X$, then Kodaira's  map sends $A$ and $B$ to different linear functions on $W$, since their  values at the function $f_{A,B}$  are different.
\end{proof}

\begin{definition} Assume that the space $W$ contains a constant function $f\equiv \mu$ for any $\mu \in \bold k$. A {\sl characteristic hyperplane} $H\in W^*$ is a hyperplane consisting of all linear functions  $f\in W^*$ on  $W$ whose value at the function $f\equiv 1$ is equal to $1$.
\end{definition}

\begin{lemma}\label{characteristic} Assume that the space  $W$ contains all constant functions on $X$. Then Kodaira's map sends $X$ to the characteristic hyperplane of the space $W^*$.
\end{lemma}

\begin{proof} The value of the function $f\equiv 1$ at each point $A\in X$ is equal to $1$.
\end{proof}

Set $X=L$, an affine space, and $W=P_1(W)$.

\begin{lemma}\label{affine}  Kodaira's map $L\to P_1^*(L)$ is an affine map.
\end{lemma}

\begin{proof} A polynomial $P$ of degree $\leq 1$ is an affine map $P:L\to\bold k^1$. Its linear part $\overline P$ is a linear function on the  space $\overline L$ of free vectors in $L$. Thus, the difference $P(B)-P(A)$ is invariant under all translations of the pair $(A,B)$ (i.e., if $(A,B)\sim(C,D)$, then
$P(B)-P(A)=P(D)-P(C)$) and depends linearly on the free vector $(A,B)$. These properties imply that  Kodaira's  map $L\to P_1(L)^*$ is an affine map.
\end{proof}

\begin{theorem}\label{affine embedding}  Kodaira's map $L\to P_1^*(L)$ provides a canonical affine embedding of the affine space $L$ to the vector space $P_1^*(L)$. The image of   $L$ under the embedding coincides with the characteristic hyperplane $H$ in the space  $P_1^*(L)$.
\end{theorem}

\begin{proof} By Lemma \ref{embedding}, the map $L\to P_1^*(L)$ is an embedding, since all constant functions on $L$ are polynomial of degree $\leq 1$. By Lemma \ref{affine},    the  map $L\to P_1^*(L)$ is affine; by Lemma \ref{characteristic}, it maps $L$ to the characteristic hyperplane $H\in P_1^*(L)$. The dimension $\dim_{\bold k} H$ of the characteristic hyperplane is equal to the dimension $\dim_{\bold k}L$ of $L$. Indeed,
$\dim _{\bold k}H=\dim_{\bold k}P_1(L)-1=\dim_{\bold k}L$. Thus, the map $L\to P_1^*(L)$ provides an isomorphism between the affine spaces  $L$ and $H$.
\end{proof}

\section{Several Interpretations of the Space of Weighty Points and Mass Dipoles}\label{sec10}

In this section, we show first that an affine map $F:L_1\to L_2$ induces the linear
map $F_*:\widehat M(L_1)\to \widehat M(L_2)$. Then we give three interpretations of the
space $\widehat M(L)$.

\subsection{A Linear Map of the Space of Weighty Points and Mass Dipoles Induced by an Affine Map}\label{subsec10.1}

Let $L_1,L_2$  be affine spaces  over a field $\bold k$, and let $D(L_1)$, $D(L_2)$  be vector spaces  over $\bold k$ of sets of weighted points in $L_1$ and $L_2$ correspondingly. Consider any map  $F:L_1\to L_2$.

\begin{definition} The  map $F_*: D(L_1)\to D(L_2)$ {\sl induced by the map} $F:L_1\to L_2$ is the map
that sends any single point $A\in L_1$, equipped with  mass 1, to the single point $F(A)\in L_2$ equipped with  mass 1, and which extends the above map by linearity to the  space $D(L)$.
\end{definition}

The definition implies the total mass of the set $\widetilde A$ in  $L_1$  and $F_*(\widetilde A)$ in $L_2$ are equal. Thus, the following inclusion holds: $F_*(D_0(L_1)\subset D_0(L_2).$

The definition of the map $F_*: D(L_1)\to D(L_2)$ does not use affine structures on $L_1$ and $L_2$ (and it could be applied to vector spaces of   weighted points in any sets $L_1$ and $L_2$). The following Lemma  uses the affine structures on $L_1$ and $L_2$ in a very essential way.

\begin{theorem}\label{theorem10.1} If $F:L_1\to L_2$ is an affine map, then  $F_*$ takes the null sets  in $L_1$ to   null sets in~$L_2$.
\end{theorem}

\begin{proof} Let $\widetilde A$ be a set of points $A_1,\dots,A_N$ equipped with masses $\lambda_1,\dots, \lambda_n$. Since $\widetilde A$ is  a null set, its total mass is equal to zero, and, for any point $O\in L_1$, the free vector $\sum \lambda_i(O,A_i)$ is equal to zero. Thus, $\overline F(\sum \lambda_i(O,A_i))$ is the zero free vector in $L_2$. It implies that the free vector  $\sum \lambda_i (F(O), F(A_i))$ equals to zero. Thus, the moment of the set $F_*(\widetilde A)$ about the point $F(O)$ equals to zero. Since the total mass of the set $F_*(\widetilde A)$ is zero, the set  $F_*(\widetilde A)$ is a null set.
\end{proof}

\begin{corollary}\label{map} If $F:L_1\to L_2$ is an affine map, then the map $F_*:\widehat M(L_1)\to \widehat M(L_2)$ is a well defined linear map.
\end{corollary}

\begin{proof} By Theorem \ref{theorem10.1}, the linear map $F_*:D(L_1)\to D(L_2)$ takes the subspace $DM(L_1)$ of null sets in $L_1$ to the subspace $DM(L_2)$ of null sets in $L_2$. Thus, $F_*$ induces a well defined linear map from the factor-space $ D(L_1)/DM(L_1)=\widehat M(L_1)$  to the  factor-space $ D(L_2)/DM(L_2)=\widehat M(L_2)$.
\end{proof}

\begin{corollary}\label{corollary10.2} Let $F:L_1\to L_2$ be an affine map.  If $\widetilde A$ is a weighty  set in $L_1$ with nonzero total mass $\lambda$ and the center of mass $O$, then the set $F_*(\widetilde A)$ is
a weighty  set in $L_2$ with nonzero total mass $\lambda$ and the center of mass $F(O)$.

If $\widetilde A$ is a weightless  set in $L_1$ whose total mass equals to zero and whose mass dipole is $\{-O,B\}$,  then the set $F_*(\widetilde A)$ is a weightless   set in $L_2$ whose total mass equals to zero and whose mass dipole  is $\{-F(O),F(B)\}$.
\end{corollary}

\begin{proof} If two weighted sets in $L_1$ are equivalent in $L_1$, then their images under the map $F_*$ are equivalent in $L_2$, since $F_*$ maps null sets in $L_1$ to null sets in $L_2$.

If a weighty set $\widetilde A$ in $L_1$ is equivalent to a weighty point $\{(O,\lambda)\}$, then the set $F_*(\widetilde A)$ in $L_2$ is equivalent to the weighty point $\{(F(O),\lambda)\}$. Thus, the first statement of the  Corollary is proved.

If a weightless set $\widetilde A$ in $L_1$ is equivalent to a mass dipole $\{-O,B\}$, then the set $F_*(\widetilde A)$ in $L_2$ is equivalent to the mass dipole $\{-F(O), F(B)\}$. Thus, the second statement of the Corollary is proved.
\end{proof}

\begin{corollary} If $F:L_1\to L_2$ is an affine embedding, then the map $F_*:\widehat M(L_1)\to \widehat M(L_2)$ is a  linear embedding.
\end{corollary}

\begin{proof} By Corollary \ref{map}, the map $F_*$ is  linear. Corollary \ref{corollary10.2} implies that if $F$ is an embedding, then $F_*$ sends nonzero elements of the space $\widehat M(L_1)$ to nonzero elements of the space $\widehat L_2$.
\end{proof}

\subsection{The Space $\widehat M(L)$ of  an Affine Hyperplane $ L$ in  a Linear Ambient Space $\bold L$}\label{subsec10.2}

In this section, we will consider an affine embedding $F:L\to \bold L$,   where $L$ is  an affine hyperplane in a vector space $\bold L$ not containing the origin $\bold O$ of the space $\bold L$. We will show that the vector space $\widehat M(L)$ of weighty points and mass dipoles in $L$ is naturally isomorphic to  $\bold L$.

Let $\widetilde L$  be  the vector subspace in $ \bold L$ parallel to the affine hyperplane $L$. It is easy to verify the following lemma.

\begin{lemma} If an affine  hyperplane $L\subset \bold L$ does not contain the origin $\bold O\in \bold L$, then there is  a unique linear function  $T_L:\bold L\to \bold k$ that is identically equal to $1 \in \bold k$ on the hyperplane~$L$.

Moreover, the function $T_L$ vanishes on  the vector space $\widetilde L\subset \bold L$ parallel to the affine hyperplane~$L$.
\end{lemma}

To provide an isomorphism between the vector spaces $\widehat M(L)=D(L)/DM(L)$ and $\bold L$ we define a linear map  $\Psi$ of the space $D(L)$ to the ambient vector space $\bold L$ whose kernel is the space $DM(L)$ of null sets in $L$.

\begin{definition} Let $D(\bold L)$ be the space of weighted points in a vector space $\bold L$. Then the {\sl evaluation map} $\Phi:D(\bold L)\to \bold L$ is the map, which sends a weighty set $\widetilde A=\{(A_i,\lambda_i)\}$ in $\bold  L$ to the vector $\sum \lambda_i A_i\in \bold L$.
\end{definition}

\begin{lemma} The evaluation map $\Phi:D(\bold L)\to \bold L$ sends all null sets in $D(\bold L)$ to zero.
\end{lemma}

\begin{proof} Indeed, by definition, the vector $\Phi(\widetilde A)$ is equal to the moment of the set $\widetilde A$ about the origin~$\bold O\in \bold L$.
\end{proof}

\begin{definition} Let $D( L)$ be the space of weighted points in an affine hyperplane $L\subset \bold L$. Then
 the {\sl  $\Psi$-map} is the map $\Psi:D(L)\to \bold L$ which is the composition $\Psi= \Phi  \circ F_*$ of the map $F_*:D(L)\to D(\bold L)$, induced by the embedding $F:L\to \bold L$ and the evaluation map $\Phi:D(\bold L)\to \bold L$.
\end{definition}

\begin{lemma} Let $T_L$ be the linear function on $\bold L$ that takes value $1$ on the affine hyperplane $L$. Then the total mass of a set $\widetilde A\in D(L)$ is equal to $T_L(\Psi(\widetilde A))$.
\end{lemma}

\begin{proof} Indeed, if $\widetilde A$ is a weighty point $A_1$ of mass $1$, then $T_L(\widetilde A)=T_L(\Phi \circ F_*(\widetilde A))=T_L(A_1)=1$. This relation can be extended by linearity to any set $\widetilde A=\{(A_i,\lambda_i)\}$:
$$T_L(\Psi(\widetilde A))=T_L(\sum \lambda_i F(A_i))=\sum \lambda_i T_L( A_i)=\sum \lambda_i.$$
The lemma is proven.
\end{proof}

\begin{lemma}\label{kernel} If $\Psi(\widetilde A)=0$, then  $\widetilde A$ is a null set.
\end{lemma}

\begin{proof} Let $\widetilde A$ be a weighted set $\{(A_i,\lambda_i)\}$ in $D(L)$. By definition, its image $\Psi(\widetilde A)=\sum \lambda_i A_i$ is the moment of the set $F_*(\widetilde A)\in D(\bold L)$ about the origin $\bold O\in \bold L$. Thus, if $\Psi(\widetilde A)=0$, then the weighted set $\widetilde A \in D(\bold L)$ has  zero moment about the pivot point $\bold O$. The total mass of the set $\widetilde A$ is also equal to zero, since $I_L(\Psi(\widetilde A))=T_L(0)=0$. Thus, the weighty set $\widetilde A$ is a null set in $\bold L$. The map $F:L\to \bold L$ is an embedding. Thus, the set $\widetilde A$ as a weighted set in $L$ also is a null set.
\end{proof}

\begin{theorem}\label{hyperplane and isomorphism} The map $\Psi:D(L)\to \bold L$ is an isomorphism between the space $\widehat M(L)$ of weighty points and mass dipoles of an affine hyperplane $L\subset \bold L$ not passing through the origin $\bold O\in \bold L$ and the ambient vector space $\bold L$.
\end{theorem}

\begin{proof} By Lemma \ref{kernel},  the induced map $\Psi:\widehat M(L)\to \bold L$ is an embedding. The image of $\widehat M(L)$ under the map $\Psi$, obviously, contains the hyperplane $L$. The smallest vector space that contains this hyperplane is the space $\bold L$. Thus, the induced map $\Phi:\widehat M(L)\to \bold L$ is an embedding and a surjective map. Thus, $\Psi$ is an isomorphism between the spaces $\widehat M(L)$ and $\bold L$.
\end{proof}

Consider  an affine hyperplane $L$  in a vector space $\bold L$ not passing through the origin. Let  $\widetilde A$ be a set of points $\{A_1,\dots, A_N\}$ in $L$ equipped with masses $\lambda_1,\dots,\lambda_N$. Let   $\lambda$ be the total mass of the set $\widetilde A$. Theorem \ref{hyperplane and isomorphism} implies the following corollary.

\begin{corollary} If $\lambda\neq 0$, then $\sum \lambda_i A_i=\lambda O,$
where $O$ is the center of mass of the set $\widetilde A$.

 If $\lambda=0$, then  $\sum \lambda_i A_i= B-O\in \widetilde L$,  where $\{-O,B\}$ is a mass dipole of the set $\widetilde A$, and $B-O$ is a well defined vector in the vector space $\widetilde L$ parallel to the affine space $L$.
 \end{corollary}

\subsection{The Spaces $\widehat M(L)$ and $P_1^*(L)$ are Canonically
Isomorphic}\label{subsec 10.3}

Let $L$ be an affine space over a field $\bold k$, and let $P_1(L)$ be a vector space of degree $\leq 1$ polynomials on $L$.

Recall that the Kodaira's map $Kod: L\to P_1^*(L)$ sends a point $x\in L$ to the linear function on $P_1(L)$ which assigns to a polynomial $f\in P_1(L)$ its value $f(x)\in \bold k$ at the point $x$.

In Section \ref{subsec9.2}, we showed that {\sl the map $Kod$ is  an affine embedding of  $L$ to  $P_1^*(L)$, whose image is the characteristic hyperplane {\rm (see Section \ref{subsec9.2})} in $P_1^*(L)$}.

Thus, the Kodaira's map provides a canonical representation of an affine space $L$ as the characteristic hyperplane in the vector space $P_1^*(L)$.

 \begin{corollary}\label{corollary9.6} There is a natural identification between the vector space $\widehat M(L)$ of weighty points and mass dipoles in an affine space $L$  and  the dual space $P_1^*$  to the space $P_1(L)$ of polynomials of degree $\leq 1$ on $L$.
\end{corollary}

Let  $\widetilde A$ be a set of points $\{A_1,\dots, A_N\}$ in $L$ equipped with masses $\lambda_1,\dots,\lambda_N$, and  let   $\lambda$ be the total mass of the set $\widetilde A$. Theorem \ref{hyperplane and isomorphism}  and Corollary \ref{corollary9.6} imply the following theorem.

\begin{theorem} Consider a linear function $Kod(\widetilde A)$ on the space $P_1(L)$ which sends a polynomial $f\in P_1$ to $\sum \lambda_if(A_i)\in \bold k$. If $\lambda\neq 0$, then $\sum \lambda_i f(A_i)=\lambda f(O) $,
where $O$ is the center of mass of the set $\widetilde A$ and $f$ is any polynomial  from  $ P_1(L)$.
 If $\lambda=0$, then  $\sum \lambda_i f( A_i)= f(B)-f(O)\in \widetilde L$, where $\{-O,B\}$ is a mass dipole of the set $\widetilde A$,
and $f$ is any polynomial  from the space $\in P_1(L)$.
 \end{theorem}

\subsection{Differentials of Quadratic Polynomials  and the Space of Weighty Points and Mass Dipoles}\label{subsec10.4}

Consider a vector space $\bold L$ over a field $\bold k$ as an affine space $L$. Let $B$ be a non-degenerate symmetric bilinear form on $\bold L$ and let $L_B$ be a space
of polynomials $T$ on $\bold L$,  representable in the form  $P=\lambda Q_B+l+c$, where $Q_B(x)=B(x,x)$ is the quadratic form, associated with $B$, $l\in \bold L^*$ is an arbitrary  linear function, and $\lambda,c\in \bold k$ are arbitrary constants.

The differential $DT_x$  of a polynomial $T$ on $\bold L$ give rise to the map $DT:\bold L\to \bold L^*$ which assigns to a point $x\in \bold L$ the linear function $DT_x(y)$ of $y\in \bold L$.

With the space $L_B$, one can associate the vector space $DL_B$  of differentials  of all polynomials $T\in L_B$.

\begin{theorem}\label{quadratic form} The vector space $DL_B$  of differentials of all polynomials from the space $L_B$ is isomorphic to the space $\widehat M(L)$ of weighted points and mass dipoles in $L$. Under this isomorphism:
\begin{enumerate}

\item a weighty point $O\in L$ of a  nonzero mass $\lambda\in \bold k$ corresponds to the differential of a polynomial $T=\lambda Q_B+l+c$ whose  critical point  is the point $O$, i.e.,  $T=\lambda Q_B(x-O) +c_0$, where $c_0\in \bold k$ is an arbitrary constant;

\item  a mass  dipole $\{-C,D\}$  corresponds to the differential of a polynomial $P=l+c$ of degree $\leq 1$, where  $l$ is the linear function defined by relation  $l(x)=  -2B(D-C, x)$.
\end{enumerate}
\end{theorem}

Consider a polynomial $T\in L_B$ representable in the form $$T(x)=\sum \lambda_i Q_B(x-A_i).$$
Denote by $\widetilde A$ the set of points $\{A_1,\dots, A_N\}$ equipped with masses $\lambda_1,\dots,\lambda_N$. Let   $\lambda$ be the total mass of the set $\widetilde A$. The theorem implies the following corollary.

\begin{corollary} If $\lambda\neq 0$, then $T(x) =\lambda Q_B(x-O)+c_0$,
where $O$ is the center of mass of the set $\widetilde A$ and $c_0\in\bold k$ is some constant.

 If $\lambda=0$, then  $T(x) =l(x)+c_0$,  where $l$ is the linear function defined by relation  $l(x)=  -2B(D-C, x)$, where $\{-C,D\}$ is the mass dipole of the set $\widetilde A$, and $c_0\in \bold k$ is some constant.
 \end{corollary}

The differential $DT_x$ of a polynomial $T$  provides a map from the space $\bold L$ to the space  $\bold L^*$ of linear functions on $\bold L$.

One can identify the vector space $\bold L$   with the space $\bold L^*$ by  choosing  a non-degenerate symmetric bilinear form $F$ on $\bold L$.

\begin{definition} The linear function $F$-dual to a vector $X\in \overline L$ is defined as the  linear function  $l=x^*_F \in \bold L^*$, satisfying the following identity $l(y) \equiv F(x,y)$.

A vector $F$-dual to a  linear function $l\in \bold L^*$ is  defined as  the vector $x=l^*_F \in \bold L$ such that $l=x^*_F$.
\end{definition}

\begin{definition} The{\sl  $F$-gradient}  $\nabla _F T(x)$ of a polynomial $T$ at a point $x\in L$
is the vector that is $F$-dual to the differential $DT_x$ of $T$ at  $x\in L$, i.e.,  the following identity holds:
$$ F(\nabla_F T(x),y)\equiv DT_x(y).$$
\end{definition}

One can check that
the differential $D_x T (y)$ of a polynomial $T=\lambda Q_B+l+c$  at a point $x\in \bold L$ is equal to $2\lambda B(x,y)+l(y)$ (as a function of $y$).

We will identify the space $\bold L$ with the dual space $\bold L^*$ using a symmetric bilinear form $F=-2B$.

With this choice of a bilinear  form $F$, the $F$-gradient $\nabla_F T(x)$ satisfies
the following relation:
\begin{equation}\label{graduent}
\nabla_F T(x) = -\lambda x+b,
\end{equation}
where $b=l^*_F$.

\begin{proof}[Proof of Theorem \ref{quadratic form}]
The space of differentials of polynomials  $T\in L_B$ is isomorphic to the space of $F$-gradients of polynomials $T$. Formula (\ref{graduent}) implies that the $F$-gradient of a polynomial $T\in L_B$ is a moment-like map whose total mass equals $\lambda$. If $\lambda\neq 0$, then the $F$-gradient vanishes at a single point which is the center of mass of the moment-like map and at the same time the critical point of the polynomial $T$.
If $\lambda=0$, then $F$-gradient is identically equal  to the constant   $b$ defined  above.

The space  $\widehat M(L)$ is isomorphic to the space $M(L)$ of moment-like maps. It is easy to see  that the isomorphism between $\widehat M(L)$ and $M(L)$ agrees with the statement of the theorem.
\end{proof}






{\small\bibliography{commat}}
\EditInfo{July 2, 2023}{October 3, 2023}{Jacob Mostovoy and Sergei Chmutov
}
\end{document}